\documentclass[a4paper,10pt]{article}

\usepackage{amsmath,amssymb,amsthm}
\usepackage{amssymb,latexsym, bbm,comment}

\numberwithin{equation}{section}

\newcommand{\origsetminus}{} \let\origsetminus=\setminus           
\renewcommand{\setminus}{\!\origsetminus\!}

\let\oldmarginpar\marginpar
\renewcommand\marginpar[1]{\-\oldmarginpar[\raggedleft\footnotesize #1]%
{\raggedright\footnotesize #1}}

{\theoremstyle{plain}
\newtheorem{lemma}{Lemma}[section]
\newtheorem{theorem}[lemma]{Theorem}
\newtheorem{corollary}[lemma]{Corollary}

} {\theoremstyle{definition}
\newtheorem{definition}[lemma]{Definition}
\newtheorem{example}[lemma]{Example}
\newtheorem{remark}[lemma]{Remark}
}

\renewcommand{\mathbb}{\mathbbm}                     
\renewcommand{\epsilon}{\varepsilon}                 
\renewcommand{\phi}{\varphi}
\renewcommand{\theta}{\vartheta}
\renewcommand{\le}{\leqslant}
\renewcommand{\ge}{\geqslant}

\newcommand{\origfoo}{} \let\origfoo=\sqrt           
\renewcommand{\sqrt}[1]{\origfoo{#1}\;}

\renewcommand{\O}{{\mathcal O}}                      
\newcommand{\abs}[1]{\left\lvert #1 \right\rvert}    
\newcommand{\norm}[1]{\left\lVert #1 \right\rVert}   
\DeclareMathOperator{\F}{{\cal F}}                   
\DeclareMathOperator{\R}{{\mathbb R}}                

\DeclareMathOperator{\Rp}{{\mathbb R}_+}             
\DeclareMathOperator{\C}{{\mathbb C}}                
\DeclareMathOperator{\N}{{\mathbb N}}                
\newcommand{\A}{{\mathcal A}}

\DeclareMathOperator{\Borel}{{\mathcal B}}
\newcommand{\scapro}[2]{\langle #1,#2\rangle}       
\newcommand{\scaprob}[2]{\big\langle #1,#2\big\rangle}       
\DeclareMathOperator{\1}{\mathbbm 1}

\renewcommand{\L}{{\mathcal L}}

\newcounter{zahl}

\DeclareMathOperator{\Cc}{{\hat{\mathcal Z}}}

\DeclareMathOperator{\Z}{{\mathcal Z}}
\renewcommand{\H}{{\mathcal S}}
\DeclareMathOperator{\G}{{\mathcal G}}

\DeclareMathOperator{\leb}{\text{leb}}

\makeatletter
\newcommand*{\inlineequation}[2][]{%
  \begingroup
    \refstepcounter{equation}%
    \ifx\\#1\\%
    \else
      \label{#1}%
    \fi
    \relpenalty=10000 %
    \binoppenalty=10000 %
    \ensuremath{%
      #2%
    }%
    ~\@eqnnum
  \endgroup
}
\makeatother

\title{Ornstein-Uhlenbeck processes driven by\\ cylindrical L{\'e}vy processes}

\author{ Markus Riedle\footnote{markus.riedle@kcl.ac.uk}{}\\
Department of Mathematics\\
King's College\\
London WC2R 2LS\\
United Kingdom }

\begin{document}

\maketitle

\begin{abstract}
In this article we introduce a theory of integration for deterministic, operator-valued
integrands with respect to cylindrical L{\'e}vy processes in separable Banach spaces. Here, a cylindrical L{\'e}vy process is understood in the classical framework of cylindrical random variables and cylindrical measures, and thus, it can be considered as a natural generalisation of cylindrical Wiener processes or white noises. Depending on the underlying Banach space, we provide necessary and/or sufficient conditions for a function to be integrable. In the last part, the developed theory is applied to define Ornstein-Uhlenbeck processes driven by cylindrical L{\'e}vy processes and several examples are considered.
\end{abstract}


\section{Introduction}

The degree of freedom of models in infinite dimensions is often reflected by the constraint
that each mode along a one-dimensional subspace is independently perturbed by the noise. In the Gaussian
setting, this leads to the {\em cylindrical Wiener process } including from a modeling point
of view the very important possibility to describe a Gaussian noise in both time and space
with a great flexibility, i.e.\ space-time white noise. Up to very recently, there has been no
analogue  for L{\'e}vy processes. The notion {\em cylindrical L{\'e}vy process} appears
the first time in the monograph \cite{PeszatZab} by Peszat and Zabczyk  and it is followed by the works Brze\'zniak et al \cite{Brzetal}, Brze\'zniak and Zabzcyk
\cite{BrzZab10}, Liu and Zhai \cite{LiuZhai}, Peszat and Zabczyk \cite{PeszatZab12} and Priola and Zabczyk \cite{PriolaZabczyk}. The first systematic
introduction of cylindrical L{\'e}vy processes appears in our work Applebaum and Riedle \cite{DaveMarkus}. In this work cylindrical L{\'e}vy processes are introduced as a natural generalisation of cylindrical Wiener processes, and they model a very general, discontinuous noise occurring in the time and state space.

The aforementioned literature (\cite{Brzetal}, \cite{BrzZab10}, \cite{LiuZhai}, \cite{PeszatZab12}, \cite{PriolaZabczyk}) study stochastic evolution equations
of the form
\begin{align}\label{eq.Cauchyintro}
  dY(t)=AY(t)\,dt + dL(t)\qquad\text{for all }t\ge 0,
\end{align}
where $A$ is the generator of a strongly continuous semigroup on a Hilbert or Banach space.
The driving noise $L$ differs in these publications but it is always constructed in an explicit way and it is referred to by the authors as {\em L{\'e}vy white noise}, {\em cylindrical stable process} or just {\em L{\'e}vy noise}. The works have in common that the solution of \eqref{eq.Cauchyintro} is represented by a stochastic convolution integral, which is based either  on the one-dimensional integration theory, if the setting allows as for instance in \cite{Brzetal}, \cite{PeszatZab12}, or on moment inequalities for Poisson random measures as for instance in \cite{BrzZab10}. However, these approaches and the results are tailored to the specific kind of noise under consideration, respectively.

The main objective of our work is to develop a general theory of stochastic integration for deterministic integrands, which provides a unified framework for the aforementioned works. Although not part of this work, the results are expected to lead to a better understanding of phenomena, which are individually observed  for the solutions of \eqref{eq.Cauchyintro} in the various models considered in the literature, such as irregularity of trajectories in \cite{Brzetal}. In order to be able to develop a general theory, we define cylindrical L{\'e}vy processes by following the classical approach to cylindrical measures and cylindrical processes, which is presented for example in Badrikian \cite{Badrikian} or Schwartz \cite{Schwartz}. This systematic approach for cylindrical L{\'e}vy processes is developed in our work together with Applebaum in \cite{DaveMarkus}. In the current work, we illustrate  that those kinds of cylindrical L{\'e}vy noise, considered in the literature, are specific examples of a cylindrical L{\'e}vy process in our approach.

Integration for random integrands with respect to other cylindrical processes than the cylindrical Wiener process is only considered in a few works. In fact, we are only aware of two approaches to integration with respect to cylindrical martingales, which originate either in the work developed by  M{\'e}tivier  and Pellaumail in  \cite{MetivierPellcylindrical} and \cite{MetivierPell} or by Mikulevi\v{c}ius and Rozovski\v{\i} in \cite{MikRoz98} and \cite{MikRoz99}. However,  both constructions  heavily rely on the assumption of finite weak second moments and are therefore not applicable in our framework.  For cylindrical L{\'e}vy processes with weak second moments, a straightforward integration theory is introduced in Riedle \cite{Riedle12}. It is worth mentioning that a localising procedure cannot be applied to cylindrical L{\'e}vy processes since they do not necessarily  attain values in the underlying space.

Our work can be seen as a generalisation of the publications  by
Chojnowska-Michalik \cite{Anna} on the one hand and  by Brze\'zniak and van Neerven \cite{BrzezniakvanNeerven-studia} and  by van Neerven and Weis \cite{vanNeervenWeis} on the other hand. In \cite{Anna} the author introduces a stochastic integral for deterministic integrands with respect to genuine L{\'e}vy processes in Hilbert spaces. If we apply our approach to this specific setting, the class of admissible integrands for the integral, developed in our work, is much larger than the one in \cite{Anna}, see Remark \ref{re.anna}. The articles \cite{BrzezniakvanNeerven-studia} and \cite{vanNeervenWeis} introduce a stochastic integral with respect to a cylindrical Wiener processes in Banach spaces. But the approach in  \cite{BrzezniakvanNeerven-studia} and \cite{vanNeervenWeis} differs from ours as the Gaussian distribution enables the authors to rely on an isometry
in terms of square means.

Our approach to develop a stochastic integral is based on the idea to introduce first a
cylindrical integral, which exists under mild conditions, since cylindrical random variables are  more general
objects than genuine random variables. A function is then called  stochastically integrable if its cylindrical integral is actually induced by a genuine random variable in the underlying Banach space, which is then called the stochastic integral. The advantage of this approach is that
the latter step is a purely measure-theoretical problem, which can be formulated in terms of the characteristic function of the cylindrical integral. Since the stochastic integral, if it exists, must be infinitely divisible, we can describe the class of integrable functions by using conditions which guarantee the existence of an infinitely divisible random variable for a given  semimartingale characteristics.
The candidate for the semimartingle characteristics is provided by the cylindrical integral, since its characteristic function coincides with the one of its possible extension to a genuine random variable. Conditions, guaranteeing the existence of an infinitely divisible measure in terms of the characteristic function, are known in many spaces, e.g. in  Hilbert spaces and in Banach spaces of Rademacher type. The developed theory of stochastic integration is applied to define Ornstein-Uhlenbeck processes driven by cylindrical L{\'e}vy processes. We show in two corollaries, that our approach easily  recovers results from the literature on the existence of Ornstein-Uhlenbeck processes.

The organisation of this paper is as follows. In Section \ref{se.preliminaries} we collect some
preliminaries on cylindrical measures and cylindrical random variables, which can be found for example in Badrikian \cite{Badrikian} and Schwartz \cite{Schwartz}. In Section 3 we recall the definition of a cylindrical L{\'e}vy process, based on our paper \cite{DaveMarkus} with Applebaum, and we cite some results on its characteristics from
Riedle \cite{Riedle11}. Furthermore, some important properties of cylindrical L{\'e}vy processes are established. Section 4 provides several examples of cylindrical L{\'e}vy processes. In particular, we show that the noises, considered in the aforementioned publications, are covered by our systematic approach. In Section 5 we develop a theory of stochastic integration for deterministic, operator-valued integrands with respect to cylindrical L{\'e}vy processes. We finish this work with Section 6 where we apply the developed  integration theory to treat Ornstein-Uhlenbeck processes.

\section{Preliminaries}\label{se.preliminaries}

Let $U$ be a separable Banach space with dual $U^\ast$. The dual pairing is denoted by
$\scapro{u}{u^\ast}$ for $u\in U$ and $u^\ast\in U^\ast$. The Borel $\sigma$-algebra in $U$ is
denoted by  $\Borel(U)$ and the closed unit ball at the origin by $B_U:=\{u\in U:\, \norm{u}\le 1\}$. The space of positive, finite Borel measures on $\Borel(U)$ is denoted by $M(U)$
and it is equipped with the the topology of weak convergence. The Bochner space is denoted by
$L^1([0,T];U)$ and it is equipped with the standard norm.

For every $u^\ast_1,\dots, u^\ast_n\in U^{\ast}$ and $n\in\N$ we define a linear map
\begin{align*}
  \pi_{u^\ast_1,\dots, u^\ast_n}\colon U\to \R^n,\qquad
   \pi_{u^\ast_1,\dots, u^\ast_n}(u)=(\scapro{u}{u^\ast_1},\dots,\scapro{u}{u^\ast_n}).
\end{align*}
Let $\Gamma$ be a subset of $U^\ast$. Sets of the form
 \begin{align*}
C(u^\ast_1,\dots ,u^\ast_n;B):&= \{u\in U:\, (\scapro{u}{u^\ast_1},\dots,
 \scapro{u}{u^\ast_n})\in B\}\\
 &= \pi^{-1}_{u^\ast_1,\dots, u^\ast_n}(B),
\end{align*}
where $u^\ast_1,\dots, u^\ast_n\in \Gamma$ and $B\in \Borel(\R^n)$ are called {\em
cylindrical sets}. The set of all cylindrical sets is denoted by $\Z(U,\Gamma)$
and it is an algebra. The generated $\sigma$-algebra is denoted by
$\Cc(U,\Gamma)$ and it is called the {\em cylindrical $\sigma$-algebra with
respect to $(U,\Gamma)$}. If $\Gamma=U^\ast$ we write $\Z(U):=\Z(U,\Gamma)$ and
$\Cc(U):=\Cc(U,\Gamma)$.

A function $\eta\colon \Z(U)\to [0,\infty]$ is called a {\em cylindrical measure on
$\Z(U)$}, if for each finite subset $\Gamma\subseteq U^\ast$ the restriction of
$\eta$ to the $\sigma$-algebra $\Cc(U,\Gamma)$ is a measure. A cylindrical
measure $\eta$ is called finite if $\eta(U)<\infty$ and a cylindrical probability
measure if $\eta(U)=1$.

For every function $f\colon U\to\C$ which is measurable with respect to
$\Cc(U,\Gamma)$ for a finite subset $\Gamma\subseteq U^\ast$ the integral $\int
f(u)\,\eta(du)$ is well defined as a complex valued Lebesgue integral if it
exists. In particular, the characteristic function $\phi_\eta\colon U^\ast\to\C$ of a
finite cylindrical measure $\eta$ is defined by
\begin{align*}
 \phi_{\eta}(u^\ast):=\int_U e^{i\scapro{u}{u^\ast}}\,\eta(du)\qquad\text{for all }u^\ast\in
 U^\ast.
\end{align*}

Let $(\Omega,\A,P)$ be a probability space. The space of equivalence classes of measurable functions $f\colon\Omega\to U$ is denoted by $L_P^0(\Omega;U)$ and it is equipped with the  topology of convergence in probability.

Similarly to the correspondence between measures and random variables there is an analogous random object associated to cylindrical measures: a {\em cylindrical random variable $Z$ in $U$} is a linear and continuous map
\begin{align*}
 Z\colon U^\ast \to L^0_P(\Omega;\R).
\end{align*}
Here, continuity is with respect to the norm topology on $U^\ast$ and the topology of convergence in probability. A family $(Z(t):\,t\ge 0)$ of cylindrical random variables $Z(t)$ is called a {\em cylindrical process}.
The characteristic function of a cylindrical random  variable $Z$ is
defined by
\begin{align*}
 \phi_Z\colon U^\ast \to\C, \qquad \phi_Z(u^\ast)=E\big[\exp(iZu^\ast)\big].
\end{align*}
If $C=C(u_1^\ast,\dots, u_n^\ast;B)$ is a cylindrical set for
$u^\ast_1,\dots, u^\ast_n\in U^\ast$ and $B\in \Borel(\R^n)$ we obtain a cylindrical probability measure $\eta$ by the prescription
\begin{align*}
  \eta(C):=P\big((Zu^\ast_1,\dots, Zu^\ast_n)\in B\big).
\end{align*}
We call $\eta$ the {\em cylindrical distribution of $Z$} and the
characteristic functions $\phi_\eta$ and $\phi_Z$ of $\eta$ and $Z$
coincide. Conversely, for every cylindrical probability measure $\eta$ on
$\Z(U)$ there exist a probability space $(\Omega,\A,P)$ and a
cylindrical random variable $Z\colon U^\ast\to L^0_P(\Omega;\R)$ such
that  $\eta$ is the cylindrical distribution of $Z$, see
\cite[VI.3.2]{Vaketal}.

Let $\theta$ be an infinitely divisible probability measure on $\Borel(U)$. Then the characteristic function  $\phi_{\theta}\colon U^\ast\to\C$ of $\theta$ is given for each $u^\ast\in U^\ast$  by
\begin{align} \label{eq.charLevy}
  \phi_{\theta}(u^\ast)=\exp\left( i\scapro{b}{u^\ast}-\tfrac{1}{2} \scapro{Ru^\ast}{u^\ast}
   +\int_U\left(e^{i\scapro{u}{u^\ast}}-1- i\scapro{u}{u^\ast}   \1_{B_U}(u)\right)\nu(du)    \right),
\end{align}
where $b\in U$, $R\colon U^\ast \to U$ is the covariance operator of a Gaussian measure on $\Borel(U)$ and $\nu$ is a $\sigma$-finite measure on $\Borel(U)$.
Since the triplet $(b,R,\nu)$ is unique (\cite[Th.5.7.3]{Linde}), it  characterises the distribution of the probability measure $\theta$, and  it is  called the {\em characteristics of  $\theta$}. If $X$ is an $U$-valued random variable which is infinitely divisible, then we call the characteristics of its probability distribution the characteristics of $X$.

In general Banach spaces it is not as straightforward to define a L{\'e}vy measure as in Hilbert spaces.
In this work we use the following result (Theorem 5.4.8 in Linde [12]) as the definition:
a $\sigma$-finite measure $\nu$ on a Banach space $U$ is called a {\em L{\'e}vy measure} if
\begin{enumerate}
  \item[(i)]  $\displaystyle \int_U\big(\scapro{u}{u^\ast}^2\wedge 1\big)\,\nu(du)<\infty$ for all $u^\ast\in U^\ast$;
  \item[(ii)] there exists a measure on $\Borel(U)$ with characteristic function
\begin{align*}
  \phi(u^\ast)=\exp\left(\int_U \left( e^{i\scapro{u}{u^\ast}}-1-i\scapro{u}{u^\ast}\1_{B_U}(u)\right)
  \,\nu(du)\right).
\end{align*}
\end{enumerate}

Let $\{\F_t\}_{t\ge 0}$ be a filtration for the probability space
$(\Omega,\A,P)$. An adapted, stochastic process $L:=(L(t):\, t\ge 0)$ with values in $U$ is called a {\em L{\'e}vy process} if $L(0)=0$ $P$-a.s., $L$ has independent and stationary increments and  $L$ is continuous in probability. It follows 
that there exists a version of $L$ with paths which are continuous from the right and have limits from the left (c{\`a}dl{\`a}g paths). The random variable $L(1)$ is infinitely divisible and we call its characteristics the characteristics of $L$.

\section{Cylindrical L{\'e}vy processes}

Let $U$ be a separable Banach space. A cylindrical probability measure $\eta$ on $\Z(U)$ is called {\em infinitely divisible} if for each $k\in\N$ there exists a cylindrical probability measure $\eta_k$ such that $\eta=\big(\eta_k)^{\ast k}$. In \cite{DaveMarkus} and \cite{Riedle11}
we show that the characteristic function  $\phi_{\eta}\colon U^\ast\to\C$ of $\eta$ can be represented by
\begin{align}\label{eq.charcylLevy}
\begin{split}
    \phi_{\eta}(u^\ast)&=\exp\left(i a(u^\ast) -\tfrac{1}{2} qu^\ast
   +\int_U\left(e^{i\scapro{u}{u^\ast}}-1- i\scapro{u}{u^\ast}   \1_{B_{\R}}(\scapro{u}{u^\ast})\right)\mu(du) \right)  \\
  &=:\exp\big(\Psi(u^\ast)\big),
\end{split}
\end{align}
where $a\colon U^\ast\to\R$ is a mapping with $a(0)=0$ and which is continuous on finite dimensional subspaces,  $q\colon U^\ast \to \R$ is a quadratic form and $\mu$ is a cylindrical measure on $\Z(U)$ satisfying
\begin{align}\label{eq.cylLevymeasure}
  \int_U \big(\scapro{u}{u^\ast}^2 \wedge 1\Big) \,\mu(du)<\infty
  \qquad\text{for all }u^\ast\in U^\ast.
\end{align}
Consequently, the triplet $(a,q,\mu)$ characterises the distribution of the cylindrical measure $\eta$ and thus, it is  called the {\em (cylindrical) characteristics of $\eta$}. The mapping $\Psi\colon U^\ast\to \C$ is called the {\em (cylindrical)  symbol} of $\eta$.

We call a cylindrical measure $\mu$ on $\Z(U)$ a {\em (cylindrical) L{\'e}vy measure } if it satisfies \eqref{eq.cylLevymeasure}. However, note that it is not sufficient for a cylindrical measure $\mu$ to satisfy \eqref{eq.cylLevymeasure} in order to guarantee that there exists a corresponding infinitely divisible cylindrical measure with characteristics $(0,0,\mu)$, see \cite{Riedle11}. For a cylindrical or classical L{\'e}vy measure $\mu$ we denote $\mu^-(C):=\mu(-C)$ for all $C\in\Z(U)$.

A cylindrical process $(L(t):\,t\ge 0)$ is called a {\em cylindrical L{\'e}vy process in $U$} if
for all $u_1^\ast,\dots, u_n^\ast\in U^\ast$ and $n\in\N$ we have that
\begin{align*}
  \big((L(t)u_1^\ast,\dots, L(t)u_n^\ast):\, t\ge 0\big)
\end{align*}
is a L{\'e}vy process in $\R^n$. This definition is introduced in our work \cite{DaveMarkus}. It follows that the cylindrical distribution $\eta$ of $L(1)$ is infinitely divisible and that  the characteristic function of $L(t)$ for all $t\ge 0$ is given by
\begin{align}\label{eq.charcylLevy-Levy}
\phi_{L(t)}\colon U^\ast\to\C,\qquad  \phi_{L(t)}(u^\ast)&=\exp\big(t\Psi(u^\ast)\big),
\end{align}
where $\Psi\colon U^\ast \to \C$ is the symbol of $\eta$. We call the symbol $\Psi$ and the characteristics $(a,q,\mu)$ of $\eta$  the {\em  (cylindrical) symbol} and the {\em (cylindrical) characteristics of $L$}.

A cylindrical L{\'e}vy process $(L(t):\,t\ge 0)$ with characteristics $(a,q,\mu)$ can be decomposed into
\begin{align}\label{eq.L-decompose}
  L(t)=W(t)+P(t) \qquad\text{for all }t\ge 0,
\end{align}
where $W(t)$ and $P(t)$ are linear maps from $U^\ast$ to $L_P^0(\Omega;\R)$, see \cite[Th.3.9]{DaveMarkus}. For each $u^\ast\in U^\ast$ the stochastic processes $(W(t)u^\ast:\, t\ge 0)$ and $(P(t)u^\ast:\,t\ge 0)$ are unique up to indistinguishability by \cite[Th.I.4.18]{JacodShiryaev}.
In addition to the assumed continuity of the operator $L(t)$ we require in this work  that $W(t)$ and $P(t)$ are continuous for each $t\ge 0$,
in which case $(W(t):\,t\ge 0)$ is a cylindrical L{\'e}vy process with characteristics $(0,q,0)$ and $(P(t):\,t\ge 0)$ is an independent, cylindrical L{\'e}vy process with characteristics $(a,0,\mu)$.
Lemma 4.4.\ in \cite{Riedle11} guarantees that the characteristics $(a,q,\mu)$ obeys:
\begin{enumerate}
\item[(1)] $a\colon U^\ast\to \R \;$ is continuous;
\item[(2)] there exists a positive, symmetric operator $Q\colon U^\ast\to U^{\ast\ast}$ such that
\begin{align*}
 qu^\ast = \scapro{u^\ast}{Qu^\ast}\; \text{ for all $u^\ast \in U^\ast$};
\end{align*}
\item[(3)] for every sequence $(u_n^\ast)_{n\in\N}\subseteq U^\ast$ with $\norm{u_n^\ast - u_0^\ast}\to 0$ for some $u_0^\ast\in U^\ast$
 it follows that
 \begin{align}\label{eq.weakconvcont}
   & \big(\abs{\beta}^2\wedge 1\big)\big(\mu\circ (u_n^\ast)^{-1}\big)(d\beta)\to
  \big(\abs{\beta}^2\wedge 1\big)\big(\mu\circ (u^\ast_0)^{-1}\big)(d\beta)\;\text{ weakly in $M(\R)$.}
 \end{align}
\end{enumerate}
In this case, we replace the covariance $q$ by the covariance operator $Q$ and write $(a,Q,\mu)$ for the cylindrical characteristics.

We will need several times the following property of an arbitrary cylindrical L{\'e}vy measure in a Banach space. For classical L{\'e}vy measures, the same property can be deduced by different arguments,  see \cite[Pro.5.4.5]{Linde}.
\begin{lemma}\label{le.thelemma}
Let $\mu$ be the cylindrical L{\'e}vy measure of a cylindrical L{\'e}vy process in $U$.
Then for every $\epsilon>0$ there exists a $\delta>0$ such that
\begin{align*}
  \sup_{\norm{u^\ast}\le \delta}\int_U \Big(\abs{\scapro{u}{u^\ast}}^2\wedge 1\Big)\, \mu(du)\le \epsilon.
\end{align*}
\end{lemma}
\begin{proof}
Due to \eqref{eq.L-decompose} we can assume that the cylindrical L{\'e}vy process $L$ has the characteristics $(a,0,\mu)$. If $L^\prime$ denotes an independent copy of $L$ then the cylindrical L{\'e}vy process $\tilde{L}:=L-L^\prime$ has the  characteristics $(0,0, \mu+\mu^-)$.

Define for every $u^\ast\in U^\ast$ the cylindrical set $D(u^\ast):=\{u\in U:\, \abs{\scapro{u}{u^\ast}}\le 1\}$. The inequality $1-\cos(\beta)\ge \tfrac{1}{3}\beta^2$ for all $\abs{\beta}\le 1$ implies by using the symmetry of $\mu+\mu^{-}$, that the characteristic function of $\tilde{L}(1)$ satisfies  for each $u^\ast\in U^\ast$:
\begin{align*}
  \phi_{\tilde{L}(1)}(u^\ast)
  &= \exp\left(-\int_U \big(1- \cos(\scapro{u}{u^\ast})\big)\, (\mu+\mu^-)(du)\right)\\
& \le  \exp\left(-\int_{D(u^\ast)}\big( 1- \cos(\scapro{u}{u^\ast})\big)\, (\mu+\mu^-)(du)\right)\\
& \le  \exp\left(-\tfrac{2}{3}\int_{D(u^\ast)} \abs{\scapro{u}{u^\ast}}^2\, \mu(du)\right).
\end{align*}
Consequently, we obtain
\begin{align*}
  \int_{D(u^\ast)} \abs{\scapro{u}{u^\ast}}^2 \mu(du)
  \le -\tfrac{3}{2}\ln (\phi_{\tilde{L}(1)}(u^\ast))
  \qquad \text{for all }u^\ast\in U^\ast.
\end{align*}
Since $\tilde{L}(1)\colon U^\ast\to L^0_P(\Omega;\R)$ is continuous, its characteristic function $\phi_{\tilde{L}(1)}\colon U^\ast\to \R$ is continuous, see \cite[Pro.IV.3.4]{Vaketal}.
Therefore, there exists a $\delta_1>0$ such that
\begin{align}\label{eq.charandD}
 \sup_{\norm{u^\ast}\le \delta_1} \int_{D(u^\ast)} \abs{\scapro{u}{u^\ast}}^2 \mu(du)<\epsilon.
\end{align}

For the second part of the proof, we define $d(u^\ast):=\mu\big( D(u^\ast)^c\big)$
for all $u^\ast\in U^\ast$,
and we show that for every $\epsilon>0$ there exists a $\delta_2>0$ such that
\begin{align}\label{eq.charandDc}
  \sup_{\norm{u^\ast}\le \delta_2}d(u^\ast)\le\epsilon.
\end{align}
Assume for a contradiction that \eqref{eq.charandDc} is not satisfied. Then there exists a sequence $(u_n^\ast)_{n\in\N}\subseteq U^\ast$ with $u_n^\ast\to 0 $ as $n\to\infty$ and $d(u^\ast_n)> \epsilon$ for all $n\in\N$. For each $n\in\N$ define the  stopping time
$\tau_{u_n^\ast}:=\inf\{t\ge 0:\, \abs{ (L(t)-L(t-))u_n^\ast}>1\}$.
Since for each $n\in \N$ the stopping time $\tau_{u_n^\ast}$ is exponentially distributed with parameter $d(u^\ast_n)$ it follows that
\begin{align}\label{eq.contrad}
 P\Big( \sup_{t\in [0,T]} \abs{L(t)u^\ast_n}\le \tfrac{1}{2}\Big)
\le P\big(\tau_{u_n^\ast}>T\big)
= e^{-d(u_n^\ast)T} < e^{-\epsilon T} \quad\text{for all $n\in\N$}.
\end{align}
Let $D([0,T];\R)$ denote the space of functions on $[0,T]$ with c{\`a}dl{\`a}g trajectories and endow this space with the supremum norm. Define the mapping $L\colon U^\ast\to L^0(\Omega;D([0,T];\R))$ by $Lu^\ast:=(L(t)u^\ast:\,t\in [0,T])$. It follows by the closed graph theorem for $F$-spaces (see \cite[Th.II.6.1]{Yosida}), that $L$ is a continuous mapping. Consequently, we obtain $\sup_{t\in [0,T]} L(t)u_n^\ast\to 0$ in probability as $n\to\infty$,
which contradicts \eqref{eq.contrad}. Thus, we  have proved equality \eqref{eq.charandDc}, which together with \eqref{eq.charandD} completes the proof.
\end{proof}

The first application of the previous Lemma \ref{le.thelemma} establishes that the cylindrical symbol $\Psi$ maps bounded sets into bounded sets.
\begin{lemma}\label{le.Psibounded}
The cylindrical symbol $\Psi$ satisfies for each $c>0$:
\begin{align*}
  \sup_{\norm{u^\ast}\le c} \abs{\Psi(u^\ast)}<\infty.
\end{align*}
\end{lemma}
\begin{proof}
Let $L$ be a cylindrical L{\'e}vy process with symbol $\Psi$ and characteristics $(a,Q,\mu)$, so that $\Psi$ is of the form \eqref{eq.charcylLevy}.
Since $L(1)(\beta u^\ast)$ and $\beta L(1)u^\ast$ are identically distributed
for every $u^\ast\in U^\ast$ and $\beta>0$, equating  their L{\'e}vy-Khintchine formula yields
\begin{align}\label{eq.pandp}
  a(\beta u^\ast)= \beta a(u^\ast)+ \beta\int_U\scapro{u}{u^\ast}\Big(\1_{B_{\R}}(\beta\scapro{u}{u^\ast})-\1_{B_{\R}}(\scapro{u}{u^\ast})\Big)
     \mu(du).
\end{align}
The second term on the right hand side can be estimated by
\begin{align}\label{eq.comppandp}
&\int_U\abs{\scapro{u}{u^\ast}}\Big|\1_{ B_{\R}}(\beta \scapro{u}{u^\ast})-\1_{B_{\R}}(\scapro{u}{u^\ast})\Big|
     \mu(du)\notag\\
 &\qquad \le \int_{\abs{\scapro{u}{u^\ast}}\le\tfrac{1}{\beta}}\scapro{u}{u^\ast}^2\, \mu(du)
 + \int_{\abs{\scapro{u}{u^\ast}}>\tfrac{1}{\beta}}\, \mu(du)
 =\int_U \Big(\scapro{u}{u^\ast}^2\wedge \tfrac{1}{\beta^2}\Big)\,\mu(du).
\end{align}
The continuity of $a$ and $a(0)=0$ imply that there exists a $\delta>0$ such that $\abs{a(u^\ast)}\le 1$ for all $\norm{u^\ast}\le \delta$.
By choosing $\beta=\tfrac{c}{\delta}$ it follows from
\eqref{eq.pandp} and \eqref{eq.comppandp} by Lemma \ref{le.thelemma}, that
\begin{align}\label{eq.drift-bounded-bounded}
\sup_{\norm{u^\ast}\le c} \abs{a(u^\ast)}
 \le\beta \sup_{\norm{u^\ast}\le c}\abs{a\left(\tfrac{u^\ast}{\beta}\right)}
  +\beta \sup_{\norm{u^\ast}\le c}\int_U \Big(\abs{\scapro{u}{u^\ast}}^2\wedge \tfrac{1}{\beta^2}\Big)\,\mu(du)
  <\infty.
\end{align}
Boundedness of the term in \eqref{eq.charcylLevy} involving $Q$ can easily be established since $Q\in \L(U^\ast,U^{\ast\ast})$.
Applying the estimate
\begin{align*}
\sup_{\norm{u^\ast}\le c} \int_U\abs{e^{i\scapro{u}{u^\ast}}-1- i\scapro{u}{u^\ast}   \1_{B_{\R}}(\scapro{u}{u^\ast})}\,\mu(du)
\le 2\sup_{\norm{u^\ast}\le c} \int_U \Big(\abs{\scapro{u}{u^\ast}}^2\wedge 1\Big)\,\mu(du)
\end{align*}
completes the proof by another application of Lemma \ref{le.thelemma}.
\end{proof}

It is well known that if the covariance of a Gaussian cylindrical measure is majorised by the covariance of a Gaussian measure, it extends to a measure and is Gaussian. Next we will derive the analogue result for cylindrical L{\'e}vy measures, following the presentation of the result in the Gaussian setting:
\begin{theorem}\label{th.Gauss-dominated}
  Let $\eta$ be a centralised, Gaussian cylindrical measure on $\Z(U)$ with covariance
 $q\colon U^\ast\to \R$, i.e.
  \begin{align*}
 q(u^\ast)= \int_U \scapro{u}{u^\ast}^2\,\eta(du).
  \end{align*}
If $\theta$ is a Gaussian measure on $\Borel(U)$ with covariance operator $R\colon U^\ast\to U$ satisfying
\begin{align*}
  q(u^\ast)\le \scapro{u^\ast}{Ru^\ast}
  \qquad\text{for all }u^\ast\in U^\ast,
\end{align*}
then $\eta$ extends to a measure on $\Borel(U)$ and the extension is Gaussian.
\end{theorem}
\begin{proof}
  See Theorem 3.3.1 in \cite{Bogachev}.
\end{proof}
We extend this result to cylindrical L{\'e}vy measures by generalising Prok\-horov's theorem on projective limits (\cite[Th.9.12.2]{Bogachev-Measure} or \cite[Th.VI.3.2]{Vaketal}) to $\sigma$-finite measures. We follow here the proof of \cite{FremlinGarlingHaydon} in the form as nicely rewritten in \cite{Bogachev}.

\begin{theorem}\label{th.Levy-dominated}
  Let $\mu$ be a cylindrical L{\'e}vy measure on $\Z(U)$ and $\nu$ be a L{\'e}vy measure on $\Borel(U)$ satisfying $\mu\le \nu$ on
  $\Z(U)$. Then $\mu$ extends to a $\sigma$-finite measure on $\Borel(U)$ and the extension is a L{\'e}vy measure.
\end{theorem}
\begin{proof}
Fix a $\delta>0$ and define for each  $d\in\N$  the rectangle
\begin{align*}
  R_\delta^d:=\Big\{( \beta_1,\dots,\beta_d)\in\R^d:\, \sup_{i=1,\dots, d}\abs{\beta_i}\le \delta\Big\},
\end{align*}
and $B_\delta:=\{u\in U:\, \norm{u}\le \delta\}$.
Since $U$ is separable we can choose a norming sequence $\{u_k^\ast\}_{k\in\N}\subseteq U^\ast$ with $\norm{u_k^\ast}=1$, i.e.
\begin{align*}
  \norm{u}=\sup_{k\in\N}\abs{\scapro{u}{u_k^\ast}} \qquad\text{for all }u\in U.
\end{align*}
We define for every $k\in\N$ the mapping
$  \pi_{k}\colon U\to\R^k$ by $\pi_k(u):=\big(\scapro{u}{u_1^\ast}, \dots, \scapro{u}{u_k^\ast}\big)$,
and the finite measure
\begin{align*}
  \mu_\delta^k:=\mu\circ \pi_k^{-1}|_{\R^k\setminus R_\delta^k}
  \quad\text{on } \Borel(\R^k\setminus R_\delta^k).
\end{align*}
Since $\nu$ is a L{\'e}vy measure there exists an increasing sequence $\{K_n\}_{n\in\N}$ of compact sets $K_n\subseteq U$ (see \cite[Th.5.4.8]{Linde}), such that
\begin{align*}
  \nu\Big(\big\{u\in U:\, u\in K_n^c, \, \norm{u}>\delta\big\}\Big)\le \frac{1}{n}
  \quad\text{for all } n\in\N.
\end{align*}
By denoting the constant $c_\delta:=\nu(B_\delta^c)$ define the set
\begin{align*}
  S_\delta:=\Big\{\theta\in M(U):\, \theta(U)\le c_\delta, \,\theta(K_n^c)\le \frac{1}{n}\quad\text{for all }n\in\N\Big\}.
\end{align*}
Obviously, the set $S_\delta$ is non-empty and relatively compact in $M(U)$ by Prohorov's Theorem; see
\cite[Th.I.3.6]{Vaketal}. Furthermore, for each  $k$, $n\in\N$  we obtain
\begin{align*}
\mu^k_\delta\left(\big(\pi_k(K_n)\big)^c\right)&=
\big(\mu\circ \pi_k^{-1}\big)\left((\pi_k(K_n))^c\setminus R_\delta^k\right)\\
&\le \big(\nu\circ \pi_k^{-1}\big)\left((\pi_k(K_n))^c\setminus R_\delta^k\right)\\
& =\nu\Big(\big\{u\in U:\,u\in K_n^c,\, \sup_{i=1,\dots, k}\abs{\scapro{u}{u_i^\ast}}>\delta\big\}\Big)\\
   & \le\nu\Big( \{u\in U:\,u\in K_n^c,\, \norm{u}>\delta\}\Big)\\
  &\le \frac{1}{n}.
\end{align*}
Theorem 9.1.9 in \cite{Bogachev-Measure} implies that for each $k\in\N$ there exists a measure $\theta_\delta^k\in M(U)$
such that
\begin{align*}
  \theta_\delta^k \circ \pi_{k}^{-1}=\mu_\delta^k \quad\text{on }\Borel(\R^k\setminus R_\delta^k),
\end{align*}
with $\theta_\delta^k(U)= \mu_\delta^k (\R^k\setminus R_\delta^k)$ and $\theta_\delta^k(K_n^c)\le \frac{1}{n}$ for all $n\in\N$.
Since $\mu_\delta^k (\R^k\setminus R_\delta^k)\le c_\delta$, the set
\begin{align*}
  S_\delta^k:=\big\{\theta\in \bar{S}_\delta:\, \theta\circ \pi_k^{-1}=\mu_\delta^k \text{ on } \Borel(\R^k\setminus R_\delta^k)\big\}
\end{align*}
is non-empty. For $k\le \ell$ denote by $\pi_{k\ell}$ the natural projection from $\R^\ell$ to $\R^k$. Since $\pi_{k\ell}^{-1}(\R^k\setminus R_\delta^k)\subseteq \R^\ell\setminus R_\delta^\ell$ we obtain for $\theta\in S_\delta^\ell$ that
\begin{align*}
\theta\circ \pi_k^{-1}=(\theta\circ \pi_\ell^{-1})\circ \pi_{k\ell}^{-1}
=\mu_\delta^\ell\circ \pi_{k\ell}^{-1} = (\mu\circ \pi_\ell^{-1})\circ \pi_{k\ell}^{-1}=\mu\circ \pi_k^{-1}
\qquad\text{on } \Borel(\R^k\setminus R_\delta^k),
\end{align*}
which shows $S_\delta^\ell\subseteq S_\delta^k$. Since $\bar{S}_\delta$ is compact and $S_\delta^k$ is closed for all $k\in\N$,  the nested system $\{S_\delta^k:\, k\in\N\}$ has a non-empty intersection. Thus, for each $\delta>0$ there exists a measure $\theta_\delta\in M(U)$ satisfying
\begin{align*}
  \theta_\delta\circ \pi_k^{-1}=\mu\circ \pi_k^{-1} \quad\text{on } \Borel(\R^k\setminus R_\delta^k)\quad\text{for all }k\in\N.
\end{align*}
The measure $\theta_\delta$ is uniquely determined on
$\Borel(U)\cap B_\delta^c$ since the family of sets
\begin{align*}
  \{Z\in \Z(U):\, Z=\pi_k^{-1}(B\cap \R^k\setminus R_\delta^k)\text{ for } B\in \Borel(\R^k),\, k\in \N\}
\end{align*}
generates the $\sigma$-algebra $\Borel(U)\cap B_\delta^c$ and it is closed under intersection.
Define the measure $\tilde{\theta}_{1/k}$ to be the restriction of $\theta_{1/k}$ on the disc
$\{u\in U: \tfrac{1}{k}<\norm{u}\le \tfrac{1}{k-1}\}$ for $k\ge 2$  and $\tilde{\theta}_1$ to be the restriction of $\theta_1$ to $B_1^c$. Then
\begin{align*}
  \theta:=\sum_{k=1}^\infty \tilde{\theta}_{1/k}
\end{align*}
defines a $\sigma$-finite measure $\theta$ on $\Borel(U)$ satisfying $\theta=\mu$ on $\Z(U)$.
Since $\theta\le \nu $ on $\Z(U)$ and $\Z(U)$ is a generator of $\Borel(U)$ closed under intersection,
we have $\theta\le \nu$ on $\Borel(U)$. Proposition 5.4.5 in  \cite{Linde} guarantees that $\theta$ is a L{\'e}vy measure.
\end{proof}

\section{Examples of cylindrical L{\'e}vy processes}

\begin{example}\label{ex.genuineLevy}
If $(Y(t):\,t\ge 0)$ is a genuine L{\'e}vy process with values in a Banach space $U$, then, for $t\ge 0$,
\begin{align*}
  L(t)\colon U^\ast\to L_P^0(\Omega;\R), \qquad
   L(t)u^\ast=\scapro{Y(t)}{u^\ast}
\end{align*}
defines a cylindrical L{\'e}vy process in $U$. If $(b,R,\nu)$ is the characteristics of $Y$, then the cylindrical characteristics $(a,Q,\mu)$ of $L$ is given by
\begin{align*}
a(u^\ast)= \scapro{b}{u^\ast}+\int_U \scapro{u}{u^\ast}\big( \1_{B_{\R}}(\scapro{u}{u^\ast})-\1_{B_U}(u)\big)\, \nu(du),\qquad
Q=R, \qquad \mu=\nu.
\end{align*}
The existence of the integral is derived in  Lemma \ref{le.correction-term}.

The asymmetry of the classical characteristics and the cylindrical characteristics of $Y$ is due to the fact, that in the cylindrical perspective the entry $\mu$ is only a cylindrical measure and therefore, the truncation function $u\mapsto \1_{B_U}(u)$ cannot be integrated with respect to $\mu$. A more illustrative reason is that a classical L{\'e}vy process obviously has jumps in the underlying Banach space $U$, whereas it is not clear in which space the jumps of a cylindrical L{\'e}vy process occur.
\end{example}

An appealing way to construct a cylindrical L{\'e}vy process is by a series of
real valued L{\'e}vy processes. We denote here by $\ell^p(\R)$ for $p\in [1,\infty]$ the spaces of real valued sequences.
\begin{lemma}\label{le.weakconvsum}
Let $U$ be a Hilbert space with an orthonormal basis $(e_k)_{k\in\N}$
and let $(\ell_k)_{k\in\N}$ be a sequence of independent, real valued L{\'e}vy processes with characteristics $(b_k,r_k,\nu_k)$ for $k\in\N$.
Then the following are equivalent:
\begin{enumerate}
  \item[\rm{(a)}]
  For each $(\alpha_k)_{k\in\N}\in \ell^2(\R)$ we have
\begin{enumerate}
  \item[{\rm (i)}] $\displaystyle\sum_{k=1}^\infty \1_{B_{\R}}(\alpha_k) \abs{\alpha_k}\abs{b_k+ \int_{1< \abs{\beta}\le \abs{\alpha_k}^{-1}} \beta \,\nu_k(d\beta)}<\infty;$
\item[{\rm (ii)}] $ (r_k)_{k\in\N} \in \ell^\infty(\R);$
\item[{\rm (iii)}] $\displaystyle\sum_{k=1}^\infty\int_{\R}\left(\abs{\alpha_k \beta}^2\wedge 1\right)\,\nu_k(d\beta)<\infty$.
\end{enumerate}
\item[{\rm (b)}] For each $t\ge 0$ and $u^\ast\in U^\ast$ the sum
 \begin{align}\label{eq.weakconvsum}
L(t)u^\ast:= \sum_{k=1}^\infty \scapro{e_k}{u^\ast} \ell_k(t)
\end{align}
converges $P$-a.s.
\end{enumerate}
If in this case the set $\{\phi_{\ell_k(1)}:k\in\N\}$ is equicontinuous at 0, then
$(L(t):\,t\ge 0)$ defines a cylindrical L{\'e}vy process in $U$ with cylindrical characteristics $(a,Q,\mu)$ obeying
\begin{align*}
  a(u^\ast)&= \sum_{k=1}^\infty \scapro{e_k}{u^\ast}\left( b_k +\int_{\R} \beta \big(\1_{B_{\R}}(\scapro{e_k}{u^\ast} \beta)- \1_{B_{\R}}(\beta)\big)\,\nu_k(d\beta)\right),\\
 Qu^\ast &= \sum_{k=1}^\infty \scapro{e_k}{u^\ast} r_ke_k,
 \qquad \big(\mu\circ (u^\ast)^{-1}\big)(d\beta)
 = \sum_{k=1}^\infty \left(\nu_k\circ m_k(u^\ast)^{-1}\right)(d\beta),
\end{align*}
for each $u^\ast\in U^\ast$, where  $m_k(u^\ast)\colon \R\to\R$ is defined by $m_{k}(u^\ast)(\beta)=\scapro{e_k}{u^\ast} \beta$.
\end{lemma}
\begin{proof}
Define for an arbitrary sequence  $(\alpha_k)_{k\in\N}\subseteq \R$ and $n\in\N$ the partial sum
\begin{align}\label{eq.partsum}
  S_n(t):=\sum_{k=1}^n \alpha_k \ell_k(t)\qquad\text{for all }t\ge 0.
\end{align}
It follows by using \cite[Pro.11.10]{Sato}, that $(S_n(t):\,t\ge 0)$ is a L{\'e}vy process with characteristics
\begin{align*}
  b^{(n)}:=\sum_{k=1}^n b_k^\prime,\qquad
  r^{(n)}:=\sum_{k=1}^n \alpha_k^2 r_k,\quad
  \nu^{(n)}(d\beta):=\sum_{k=1}^n \left(\nu_k\circ m_{\alpha_k}^{-1}\right)(d\beta),
\end{align*}
where  $m_{\alpha_k}\colon \R\to\R$, $m_{\alpha_k}(\beta)=\alpha_k \beta$ and the reals $b_k^\prime$ are defined by
\begin{align*}
  b_k^\prime := \alpha_k b_k +\int_{\R}\alpha_k \beta \big(\1_{B_{\R}}(\alpha_k \beta)- \1_{B_{\R}}(\beta)\big)\,\nu_k(d\beta).
\end{align*}

For establishing the implication (a) $\Rightarrow$ (b) fix $u^\ast\in U^\ast$ and
set $\alpha_k:=\scapro{e_k}{u^\ast}$. Due to Condition~(i) there  exists $b\in\R$ such that
$  \lim_{n\to\infty} b^{(n)}=b$. Conditions (ii) and (iii) guarantee that for every continuous and bounded function $f\colon \R\to\R$ we have
\begin{align*}
&  \int_{\R} f(\beta) \left(r^{(n)}\delta_0(d\beta)+ \left(\abs{\beta}^2\wedge 1\right)\, \nu^{(n)}(d\beta)\right)\notag\\
&\qquad = f(0)\sum_{k=1}^n \alpha_k^2 r_k
  + \sum_{k=1}^n \int_{\R} f(\alpha_k\beta)\left(\abs{\alpha_k\beta}^2\wedge 1\right)\, \nu_k(d\beta)\notag\\
&\qquad \to f(0)\sum_{k=1}^\infty \alpha_k^2 r_k
  + \sum_{k=1}^\infty \int_{\R} f(\alpha_k\beta)\left(\abs{\alpha_k\beta}^2\wedge 1\right)\, \nu_k(d\beta) \qquad\text{as }n\to\infty.
\end{align*}
Theorem  VII.2.9 and Remark VII.2.10 in \cite{JacodShiryaev} 
imply that the sum $S_n(t)$ converges weakly and therefore $P$-a.s.
 to an infinitely divisible random variable $L(t)u^\ast$  for each $t\ge 0$.

Conversely, it follows from (b) that the sum $S_n(1)$, defined in \eqref{eq.partsum}, converges weakly for every $(\alpha_k)_{k\in\N}\in\ell^2(\R)$.  Theorem  VII.2.9 in \cite{JacodShiryaev} implies that $b^{(n)}$ converges as $n\to\infty$, i.e.
\begin{align*}
\abs{\sum_{k=1}^\infty  \alpha_k \left(b_k+ \int_{\R} \beta \big(\1_{B_{\R}}(\alpha_k \beta)- \1_{B_{\R}}(\beta)\big)\,\nu_k(d\beta)\right)}<\infty
\end{align*}
for every $(\alpha_k)_{k\in\N}\in \ell^2(\R)$. Since for each $k\in\N$ the term in the bracket does not depend on the sign of $\alpha_k$ we can choose $\alpha_k$ such that
each summand is positive and we obtain
\begin{align*}
\sum_{k=1}^\infty  \abs{\alpha_k} \abs{b_k+ \int_{\R} \beta \big(\1_{B_{\R}}(\alpha_k \beta)- \1_{B_{\R}}(\beta)\big)\,\nu_k(d\beta)}<\infty,
\end{align*}
which yields Condition (i) as $\abs{\alpha_k}\le 1$ for sufficiently large $k$.
 Remark VII.2.10 in \cite{JacodShiryaev}  implies that for every continuous, bounded function $f\colon \R\to\R$ the sum
\begin{align*}
\sum_{k=1}^n \left(\alpha_k^2 r_k f(0)
  + \int_{\R} f(\alpha_k\beta)\left(\abs{\alpha_k\beta}^2\wedge 1\right)\, \nu_k(d\beta)\right)
\end{align*}
converges as $n\to\infty$. Since we can assume that $r_k\ge 0$ for every $k\in\N$, Conditions (ii) and (iii) are implied by choosing $f(\cdot)=1$, which completes the proof of the equivalence (a) $\Leftrightarrow$ (b).

Clearly, $L(t)\colon U^\ast\to L^0_P(\Omega;\R)$ is linear. If a sequence $(u_n^\ast)_{n\in\N}\subseteq U^\ast$ converges to 0 then $\scapro{e_k}{u_n^\ast}\to 0$ as $n\to\infty$ uniformly in $k\in\N$.
The equicontinuity of $\{\phi_{\ell_k(1)}:k\in\N\}$ implies for each $t\ge 0$ that
$\phi_{\ell_k(t)}( \scapro{e_k}{u_n^\ast})\to 1$ for $n\to\infty$  uniformly in $k\in\N$.
Thus, we obtain
\begin{align*}
 \lim_{n\to\infty}\phi_{L(t)}(u_n^\ast)
  =\lim_{n\to\infty} \prod_{k=1}^\infty \phi_{\ell_k(t)}(\scapro{e_k}{u_n^\ast})
 = \prod_{k=1}^\infty \lim_{n\to\infty}\phi_{\ell_k(t)}(\scapro{e_k}{u_n^\ast})
  =1,
\end{align*}
which shows the continuity of $L(t)$.
Finally, $L$ has weakly independent increments, that is the random variables
\begin{align*}
  \big(L(t_1)- L(t_0)\big)u_1^\ast, \dots, \big(L(t_n)- L(t_{n-1})\big)u_{n}^\ast
\end{align*}
are independent for every $0\le t_0\le \cdots \le t_n$, $u_1^\ast, \dots, u_n^\ast\in U^\ast$ and $n\in\N$. Since  $(L(t)u^\ast:\,t\ge 0)$ is
a real valued L{\'e}vy process for each $u^\ast\in U^\ast$ it follows from \cite[Le.3.8]{DaveMarkus} that $L$ is a cylindrical L{\'e}vy process in $U$.
\end{proof}

The convergence \eqref{eq.weakconvsum} is called the {\em weakly $P$-a.s.\ convergence} of the sum $\sum e_k \ell_k(t)$ for each $t\ge 0$. If there exists a random variable $Y(t)\colon \Omega\to U$ for each $t\ge 0$ such that
\begin{align*}
  Y(t)=\sum_{k=1}^\infty e_k \ell_k(t) \qquad\text{$P$-a.s. in $U$, }
\end{align*}
then the sum is called {\em strongly $P$-a.s.convergent}. Obviously, in this case, we have
$L(t)u^\ast=\scapro{Y(t)}{u^\ast}$ for every $u^\ast\in U^\ast$ and $t\ge 0$, and $(Y(t):\,t\ge 0)$ is a $U$-valued L{\'e}vy process. One can easily show a similar result to Lemma \ref{le.weakconvsum} but we skip this; a special case can be found in \cite[Th.4.13]{PeszatZab}

\begin{example}\label{ex.cylWiener}
 Let $\ell_k$ be defined by $\ell_k(\cdot):=\sigma_k w_k(\cdot)$ where $(\sigma_k)_{k\in\N}\subseteq \R$ and
$(w_k)_{k\in\N}$ is a sequence of independent, real valued standard Brownian motions. Then the series in \eqref{eq.weakconvsum} defines
a cylindrical L{\'e}vy process $L$ if and only if $(\sigma_k)_{k\in\N}\in \ell^\infty$. In this case $L$ is called a {\em cylindrical Wiener process}. Its covariance operator $Q$ is given by
\begin{align*}
  Q\colon U^\ast\to U, \qquad Qu^\ast=\sum_{k=1}^\infty \sigma_k^2 \scapro{e_k}{u^\ast}e_k.
\end{align*}
This definition of a cylindrical Wiener process is consistent with other definitions which can be found in the literature, see \cite{Riedle10}.
\end{example}

\begin{example}\label{ex.Poisson}
For a sequence $(h_k)_{k\in\N}$  of independent, real valued Poisson processes with intensity $1$
and a sequence $\sigma:=(\sigma_k)_{k\in\N}\subseteq \R$ we define $\ell_k(\cdot):=\sigma_kh_k(\cdot)$. In this case, the sum \eqref{eq.weakconvsum} defines a cylindrical L{\'e}vy process if and only if $\sigma\in \ell^2$. The sum converges strongly if and only if $\sigma\in \ell^1$. If
  $(h_k)_{k\in\N}$ is a sequence of independent, real valued compensated Poisson processes with intensity $1$,  then the sum converges weakly if and only if $\sigma\in \ell^\infty$ and strongly if and only if $\sigma\in \ell^2$.
\end{example}

\begin{example}\label{ex.PriolaZab}
  Let $(h_k)_{k\in\N}$ be a family of independent, identically distributed, real valued, standardised, symmetric, $\alpha$-stable L{\'e}vy processes $h_k$ for $\alpha\in (0,2)$. Then
  the characteristics of $h_k$ is given by $(0,0,\rho)$ with L{\'e}vy measure
  $\rho(d\beta)=\tfrac{1}{2}\abs{\beta}^{-1-\alpha} d\beta$.
For a sequence $\sigma:=(\sigma_k)_{k\in\N}\subseteq \R$ define for each $k\in\N$ the L{\'e}vy process $\ell_k(\cdot):=\sigma_k h_k(\cdot)$. Then the characteristics of $\ell_k$ is given by $(0,0,\nu_k)$ with
  \begin{align}\label{eq.stable-one}
    \nu_k:=\rho\circ m_{\sigma_k}^{-1},
  \end{align}
where $m_{\alpha}\colon \R\to\R$ is defined by $m_{\alpha}\beta=\alpha \beta$ for some $\alpha\in\R$.  With this choice of
$(\ell_k)_{k\in\N}$ it follows by Lemma \ref{le.weakconvsum} that the sum \eqref{eq.weakconvsum} defines a cylindrical L{\'e}vy process if and only if
\begin{align*}
  \sum_{k=1}^\infty \int_{\R} \Big(\abs{\alpha_k\beta}^2\wedge 1\Big)\,\nu_k(d\beta)
  =\tfrac{2}{\alpha(2-\alpha)}\sum_{k=1}^\infty  \abs{\alpha_k\sigma_k}^\alpha <\infty
\end{align*}
for every $(\alpha_k)_{k\in\N}\in\ell^2$, which is equivalent to  $\sigma\in \ell^{(2\alpha)/(2-\alpha)}$. The cylindrical L{\'e}vy process is $U$-valued, i.e. the sum converges strongly, if and only if $\sigma\in \ell^{\alpha}$.
\end{example}

\begin{remark}\label{re.PriolaZab}
The publication \cite{PriolaZabczyk} treats the cylindrical L{\'e}vy process introduced in Example \ref{ex.PriolaZab}  and it is called {\em cylindrical stable noise}. This specific example of a cylindrical L{\'e}vy process appears also  in the publications \cite{Brzetal}, \cite{LiuZhai} and \cite{PeszatZab12}.
However, since the authors do not follow the cylindrical approach, they do not require that the sum \eqref{eq.weakconvsum} is finite, i.e. they do not impose any conditions on the sequence $\sigma$. Although this is  more general, it does not match the usual framework, if one understands cylindrical L{\'e}vy processes as a generalisation of cylindrical Wiener processes. All the different definitions of cylindrical Wiener processes, one can find in the literature,  have in common that the corresponding sum
of the form \eqref{eq.weakconvsum} converges weakly as in Example \ref{ex.cylWiener}.
\end{remark}

\begin{example}\label{ex.BrzezniakZabczyk}
In \cite{BrzZab10}, the authors construct a noise by subordination which they call {\em L{\'e}vy white noise}. In the following result we define this noise in our setting and derive its cylindrical characteristics. In contrast to the original source, our cylindrical approach enables us to introduce this noise without referring to any other space than the underlying Banach space $U$, which we consider to be more natural.
\begin{lemma}\label{le.BrzZab10}
If $W$ is a cylindrical Wiener process in $U$ with covariance operator $C$
and $\ell$ is an independent, real valued subordinator with characteristics $(\alpha,0,\rho)$ then
\begin{align}\label{eq.bspZ}
  L(t)u^\ast:= W(\ell(t))u^\ast\qquad\text{for all }u^\ast\in U^\ast,\, t\ge 0,
\end{align}
defines a cylindrical L{\'e}vy process $(L(t):\,t\ge 0)$  with characteristics
  $(0,Q,\mu)$\ given by
  \begin{align*}
    Q= \alpha C,\qquad
    \mu =(\gamma\otimes \rho)\circ \kappa^{-1},
  \end{align*}
where $\gamma$ is the canonical Gaussian cylindrical measure on the reproducing kernel Hilbert space $H_C$ of $C$ with embedding $i_C\colon H_C\to U$ and
\begin{align*}
   \kappa\colon H_C\times \Rp \to U,\qquad \kappa(h,s):=\sqrt{s}i_Ch.
\end{align*}

\end{lemma}
\begin{proof}
The very definition of $L$ implies by Lemma 3.8 in  \cite{DaveMarkus} using independence of $W$ and $\ell$ that $L$ is a cylindrical L{\'e}vy process.
The characteristic function of the subordinator $\ell$  can be analytically continued, such that for each $t\ge 0$ we obtain the Laplace transform of $\ell(t)$ by
  \begin{align*}
    E\left[\exp(-\beta \ell(t))\right]
    = \exp\left(-t \tau(\beta)\right)\qquad\text{for all }\beta>0,
  \end{align*}
  where the Laplace exponent $\tau$ is defined by
  \begin{align*}
    \tau(\beta):= \alpha \beta +\int_0^\infty \left(1-e^{-\beta s}\right)\rho(ds)
    \qquad\text{for all }\beta>0,
  \end{align*}
see \cite[Th.24.11]{Sato}.
Independence of $W$ and $\ell$ implies that for each $t\ge 0$ and $u^\ast \in U^\ast$
the characteristic function $\phi_{L(t)}$ of $L(t)$ is given by
  \begin{align}\label{eq.charY}
    \phi_{L(t)}(u^\ast)
&= \int_0^\infty E\left[e^{i W(s)u^\ast}\right]\, P_{\ell (t)}(ds)\notag\\
&= \int_0^\infty e^{-\tfrac{1}{2}s \scapro{Cu^\ast}{u^\ast}}\, P_{\ell (t)}(ds)
= \exp\left(-t\tau\left(\tfrac{1}{2}\scapro{C u^\ast}{u^\ast}\right)\right).
  \end{align}
By using  $C=i_Ci_C^\ast$, $\gamma(H_C)=1$ and the symmetry of the canonical Gaussian cylindrical measure $\gamma$
we obtain
\begin{align}\label{eq.transint}
& \int_0^\infty\left( e^{-\tfrac{1}{2}s\scapro{C u^\ast}{u^\ast} }-1  \right) \,\rho(ds)\notag\\
  &\qquad\qquad= \int_0^\infty \int_{H_C} \left(e^{i\sqrt{s}\scapro{i_Ch}{u^\ast}}-1\right)\, \gamma(dh)\,\rho(ds)\notag\\
 &\qquad\qquad= \int_0^\infty \int_{H_C} \left(e^{i\sqrt{s}\scapro{i_Ch}{u^\ast}}-1- i\sqrt{s}\scapro{i_Ch}{u^\ast}\1_{B_{\R}}(\sqrt{s}\scapro{i_Ch}{u^\ast}) \right)\, \gamma(dh)\,\rho(ds)\notag\\
    &\qquad\qquad= \int_{U} \left(e^{i\scapro{u}{u^\ast}}-1-i\scapro{u}{u^\ast}\1_{B_{\R}}(\scapro{u}{u^\ast}) \right)\, \left(\left(\gamma\otimes\rho\right)\circ \kappa^{-1}\right)(du).
\end{align}
Note that $\left(\gamma\otimes\rho\right)\circ \kappa^{-1}$ is a cylindrical L{\'e}vy measure since for each $u^\ast\in U^\ast$ we have
\begin{align*}
 \int_{U} \Big(\scapro{u}{u^\ast}^2 \wedge 1 \Big)\, \left(\left(\gamma\otimes\rho\right)\circ \kappa^{-1}\right)(du)
& =\int_0^\infty \int_{H_C} \Big(s\scapro{i_Ch}{u^\ast}^2 \wedge 1 \Big)\,
\gamma(dh)\, \rho(ds)\\
& \le\int_0^\infty \left( s \int_{H_C} \scapro{i_Ch}{u^\ast}^2 \, \gamma(dh) \wedge 1 \right)\,
\rho(ds)\\
& =\int_0^\infty
  \Big( s\scapro{Cu^\ast}{u^\ast}\wedge 1\Big)\, \rho(ds)<\infty.
\end{align*}
The finiteness of the last integral is shown in \cite[Th.21.5]{Sato}.
Applying  equality \eqref{eq.transint} to the representation \eqref{eq.charY} yields that the characteristic function of $L(t)$ is of the claimed form.
\end{proof}
\end{example}

The previous example highlights an important difference between cylindrical Wiener and cylindrical L{\'e}vy processes. According to the Karhunen-Lo{\`e}ve expansion, each cylindrical Wiener processes can be represented by a sum of independent $U$-valued random variables, see for example \cite[Th.20]{Riedle10}. However, such kind of representation cannot be expected for the noise constructed  in Lemma \ref{le.BrzZab10}.

\begin{example} Another example of a cylindrical L{\'e}vy process is the {\em impulsive cylindrical process on $L^2_{\lambda}(\O;\R)$},
which is introduced in the monograph \cite{PeszatZab}. In our work \cite{DaveMarkus} we show that also this kind of a noise can be understood as a specific example of a cylindrical L{\'e}vy approach in our general approach.
\end{example}

\section{Stochastic integration}

In this section,  $U$ and $V$ are separable Banach spaces and $(L(t):\, t\ge 0)$ denotes a cylindrical L{\'e}vy process in $U$ with characteristics $(a,Q,\mu)$. Let $f\colon [0,T]\to \L(U,V)$ be a deterministic function, where $\L(U,V)$ denotes the space of linear, bounded functions from $U$ to $V$.
The aim of this section is to define a stochastic integral
\begin{align*}
 I_A:= \int_A f(s)\,dL(s)
\end{align*}
as a $V$-valued random variable for each Borel set $A\subseteq [0,T]$.

Our approach is based on the idea that the random variable $I_A$, if it exists, must be infinitely divisible and thus, its probability distribution is uniquely described  by its characteristics, say $(b_A,R_A,\nu_A)$. If we have a candidate 
for the characteristics of $I_A$ and if there are conditions known (such as in Hilbert spaces) guaranteeing the existence of an infinitely divisible random variable in terms of its prospective characteristics, then we can describe the class of integrable functions $f\colon [0,T]\to \L(U,V)$.  This approach works in every separable Banach space $V$ in which explicit conditions on the characteristics are known guaranteeing the existence of a corresponding infinitely divisible measure.

In order to have a candidate for the characteristics of $I_A$ at hand for formulating the conditions on integrability, we first introduce a cylindrical random variable $Z_A\colon V^\ast \to L_P^0(\Omega;\R)$ as a {\em cylindrical integral of $f$}.
Then we call $f$ {\em integrable with respect to $L$} if for each Borel set $A\subseteq [0,T]$ the cylindrical integral $Z_A$ is induced by a classical random variable  $I_A\colon \Omega\to V$, i.e.\
\begin{align*}
  Z_Av^\ast= \scapro{I_A}{v^\ast}
  \qquad\text{for all }v^\ast\in V^\ast.
\end{align*}
In this way one can think of $I_A$ as a stochastic Pettis integral.

For a well defined integral we expect that
\begin{align*}
\scapro{\int_A f(s)\,dL(s)}{v^\ast}=   \int_A f^\ast (s)v^\ast\, dL(s)\qquad\text{for all }v^\ast\in V^\ast.
\end{align*}
Thus, in a first step, we introduce a real valued stochastic integral for $U^\ast$-valued functions  with respect to a cylindrical L{\'e}vy process $L$.
For this purpose,  we initially consider simple $U^\ast$-valued functions.
A deterministic function $g\colon [0,T]\to U^\ast$ is called simple if there is a partition $0=t_0\le t_1\le \cdots\le t_m=T$ such that $g$ is constant on the open interval $(t_k,t_{k+1})$ for each $k=0,\dots, m-1$. The space of all simple functions is denoted by $\H([0,T];U^\ast)$ and it is endowed with the supremum norm
\begin{align*}
  \norm{g}_\infty:=\sup_{s\in [0,T]}\norm{g(s)}.
\end{align*}
Let $\G([0,T];U^\ast)$ denote the space of deterministic regulated functions; these are all functions $g\colon [0,T]\to U^\ast$ such that  for every $t\in (0,T)$ there exists the limit of $g$ on the left and on the right of $t$ and on the right of $0$ and on the left of $T$. In other words, each regulated  function has only discontinuities of the first kind. It is  shown in 
 \cite[Ch.II.1.3]{Bourbaki} or \cite[Ch.VII.6]{Dieudonne} that a function $g$ is regulated if and only if it can be uniformly approximated by step functions. In particular, regulated functions are bounded and the space $\G([0,T];U^\ast)$ can be equipped with the supremum norm, which turns it into a Banach space. 

For a simple function $g\in \H([0,T]; U^\ast)$ which attains the value $u_k^\ast$ on the interval $(t_k,t_{k+1})$ for $k=0,\dots, m-1$  define a mapping $J\colon \H([0,T];U^\ast)\to L^0_P(\Omega;\R)$ by
\begin{align}\label{eq.defI}
  J(g):=\sum_{k=0}^{m-1} \big(L(t_{k+1})-L(t_k)\big)(u_k^\ast).
\end{align}
In order to show that the mapping $J$ is continuous we need the following result.

\begin{lemma}\label{le.symbolcont}
If $\Psi\colon U^\ast\to\C$  is the symbol of a cylindrical L{\'e}vy process $L$ then the mapping
\begin{align*}
 \G([0,T];U^\ast)\to L^1([0,T];\C),
 \qquad g\mapsto\Psi(g(\cdot))
\end{align*}
is continuous.
\end{lemma}
\begin{proof}
Continuity of $\Psi$ and
Lemma \ref{le.Psibounded} guarantee that $\Psi(g(\cdot))\in L^1([0,T];\C)$ for
$g\in \G([0,T];U^\ast)$.
Let $(g_n)_{n\in\N}$ be a sequence in $\G([0,T];U^\ast)$ converging to $g\in \G([0,T];U^\ast)$.
Recall that $(a,Q,\mu)$ denotes the cylindrical characteristics of $L$ and $\Psi$ is given in \eqref{eq.charcylLevy}. Inequality  \eqref{eq.drift-bounded-bounded} shows that $a\colon U^\ast\to\R$ maps bounded sets into bounded sets and thus, Lebesgue's dominated convergence theorem implies
\begin{align*}
 \lim_{n\to\infty} \int_0^T  \abs{{a(g_n(s))}- {a(g(s))}}\,ds=0.
\end{align*}
Another application of Lebesgue's dominated convergence theorem shows
\begin{align*}
\lim_{n\to\infty} \int_0^T \abs{\scapro{g_n(s)}{Qg_n(s)}-\scapro{g(s)}{Qg(s)}}\,ds=0.
\end{align*}
Define for each $n\in\N$ the function
\begin{align*}
h_n\colon [0,T]\to \C,\qquad  h_n(s)= \int_U \left(e^{i \scapro{u}{g_n(s)}}-1- i\scapro{u}{ g_n(s)}  \1_{B_{\R}}(\scapro{u}{g_n(s)})\right) \,  \mu(du),
\end{align*}
and the function
\begin{align*}
f\colon \R\to\C,\qquad f(\beta)=
\begin{cases}
  \frac{e^{i\beta}-1-i\beta\1_{B_{\R}}(\beta)}{\beta^2\wedge 1}, & \text{if } \beta\neq 0, \\
  -\frac{1}{2}, &\text{if }\beta=0.
\end{cases}
\end{align*}
Clearly, the function $f$ is bounded and continuous.
Lemma \ref{le.thelemma}  guarantees
\begin{align*}
\sup_{n\in \N}\sup_{s\in [0,T]}\abs{h_n(s)}&= \sup_{n\in \N}\sup_{s\in [0,T]}
\abs{\int_U f(\scapro{u}{g_n(s)}) \left(\scapro{u}{g_n(s)}^2 \wedge 1\right) \,  \mu(du)}\\
&\le \norm{f}_\infty \sup_{n\in \N}\sup_{s\in [0,T]}
 \int_U \big(\scapro{u}{g_n(s)}^2 \wedge 1\big)\, \mu(du)\\
& <\infty.
\end{align*}
Since $g_n(s)\to g(s)$  for each $s\in [0,T]$, it follows from \eqref{eq.weakconvcont}:
\begin{align*}
\lim_{n\to\infty} h_n(s)&= \lim_{n\to\infty}\int_{\R}f(\beta)\big( \abs{\beta}^2\wedge 1 \big) \,
  \big(\mu\circ g_n(s)^{-1})(d\beta)\\
&= \int_{\R}f(\beta)\big( \abs{\beta}^2\wedge 1 \big) \,
  \big(\mu\circ g(s)^{-1})(d\beta)\\
&=: h(s).
\end{align*}
Lebesgue's dominated convergence theorem implies
\begin{align*}
\lim_{n\to\infty} \int_0^T \abs{h_n(s)-h(s)}\,ds=0,
\end{align*}
which completes the proof.
\end{proof}

\begin{lemma}\label{le.probcon}
  The operator $J\colon \H([0,T];U^\ast)\to L^0_P(\Omega;\R)$ defined in \eqref{eq.defI}   is continuous where $L^0_P(\Omega;\R)$ is equipped with the topology of convergence in probability.
\end{lemma}
\begin{proof}
Let $(g_n)_{n\in\N}\subseteq \H([0,T]; U^\ast)$ be a sequence converging to $g$ in $\H([0,T]; U^\ast)$.
Then, by linearity of $J$ it follows that $J(g_n)$ converges to $J(g)$ in probability  if and only if $J(g_n-g)\to 0$ in probability. However, the latter convergence occurs if and only if $J(g_n-g)\to 0$ weakly.

Independent increments of $L$ yields that the characteristic function of $J(g_n)$ is given by
\begin{align*}
\phi_{J(g_n)}\colon \R\to\C,\qquad  \phi_{J(g_n)}(\beta)= \exp\left(\int_0^T \Psi(\beta g_n(s))\,ds\right).
\end{align*}
Consequently, it follows from Lemma \ref{le.symbolcont} that
$\phi_{J(g_n)}(\beta)$ converges to $\phi_{J(g)}(\beta)$ for all $\beta\in\R$, which completes the proof.
\end{proof}

The mapping $J\colon \H([0,T]; U^\ast)\to L^0_P(\Omega;\R)$ is linear and uniformly continuous,
since the metric of $L^0_P(\Omega; \R)$ is translation invariant.
 The principle of extension by continuity (\cite[Th.I.6.17]{DunSch1}) 
enables us to extend the mapping $J$ to the space $\G([0,T];U^\ast)$, i.e.\ we define
\begin{align*}
  J(g):=\lim_{n\to\infty} J(g_n) \qquad \text{in $L^0_P(\Omega;\R)$},
\end{align*}
where $(g_n)_{n\in\N}\subseteq \H([0,T]; U^\ast)$ is chosen such that $g_n\to g$
in $\G([0,T];U^\ast)$. Lemma \ref{le.symbolcont} implies that
the characteristic function of $J(g)$ is given by
\begin{align}\label{eq.charfuncJ}
 \phi_{J(g)}\colon\R\to\C,\qquad \phi_{J(g)}(\beta)= \exp\left(\int_0^T \Psi(\beta g(s))\,ds\right).
\end{align}

Now we come back to the original aim to introduce a stochastic integral for integrands with values in  $\L(U,V)$ as a genuine $V$-valued random variable. For that purpose, let $f\colon[0,T]\to \L(U,V)$
 be a function and  for each $s\in [0,T]$ denote by $f^\ast(s)$ the adjoint of $f(s)\colon U\to V$. We say that the function $f\colon [0,T]\to \L(U,V)$  is {\em weakly in $\G([0,T];U^\ast)$}
if $f^\ast(\cdot)v^\ast$ is in $\G([0,T];U^\ast)$ for each $v^\ast\in V^\ast$.
Clearly, the function $s\mapsto f^\ast(s)v^\ast$ is measurable for each $v^\ast\in V^\ast$.

\begin{lemma}\label{le.regulated}
  If $\tau\colon [0,T]\to\R$ is measurable and bounded and $g\in \G([0,T];U^\ast)$ then
  $\tau(\cdot)g(\cdot)\in \G([0,T];U^\ast)$.
\end{lemma}
\begin{proof}
Since $\tau$ is  bounded there exist simple functions $\tau_n$ converging uniformly to $\tau$. If $(g_n)_{n\in\N}\subseteq \H([0,T]; U^\ast) $ converges uniformly to $g$ it follows
\begin{align*}
  \norm{\tau g - \tau_n g_n}_\infty
  \le \norm{\tau}_\infty \norm{g-g_n}_\infty + \norm{\tau-\tau_n}_\infty
  \norm{g_n}_\infty \to 0.
\end{align*}
Thus, $\tau g$ can be uniformly approximated by simple functions, which completes the proof.
\end{proof}

Lemma \ref{le.regulated} guarantees that if $A\in \Borel([0,T])$ and
$g\in \G([0,T];U^\ast)$ then $\1_A(\cdot) g(\cdot)\in \G([0,T];U^\ast)$.
\begin{lemma}\label{le.Zdistribution}
If $f\colon [0,T]\to\L(U,V)$ is weakly in $\G([0,T];U^\ast)$ then
for each $A\in \Borel([0,T])$
\begin{align*}
  Z_A\colon V^\ast\to L^0_P(\Omega;\R),
  \qquad Z_A v^\ast:=J\big(\1_A(\cdot)f^\ast(\cdot)v^\ast\big)
\end{align*}
defines an infinitely divisible cylindrical random variable with characteristic function
\begin{align*}
  \phi_{Z_A}\colon V^\ast\to \C, \qquad \phi_{Z_A}(v^\ast)=\exp\left(\int_A \Psi(f^\ast(s)v^\ast)\,ds\right).
\end{align*}
Furthermore, the cylindrical characteristics $(b_A,r_A,\nu_A)$ of $Z_A$ is given by
\begin{align}
  &b_A\colon V^\ast \to \R,\qquad b_A(v^\ast):=\int_A a(f^\ast(s)v^\ast)\,ds,\label{eq.3charZ_f.1}\\
  &r_A\colon V^\ast\to \R, \qquad   r_A(v^\ast)= \int_A \scapro{f^\ast(s)v^\ast}{Qf^\ast(s)v^\ast}\,ds ,\label{eq.3charZ_f.2} \\
  &\nu_A\colon \Z(V)\to [0,\infty], \qquad \nu_A = (\mu\otimes \leb)\circ \chi_A^{-1},\label{eq.3charZ_f.3}
\end{align}
where $\chi_A\colon [0,T]\times U\to V$ is defined by $\chi_A(s,u):=\1_A(s)f(s)u$.
\end{lemma}
\begin{proof}
Since  $\norm{f^\ast(\cdot)v^\ast}_\infty<\infty$ for all $v^\ast\in V^\ast$, the uniform boundedness principle implies
\begin{align}\label{eq.funiformbounded}
m:= \sup_{s\in [0,T]}\norm{f^\ast(s)}_{V^\ast\to U^\ast} <\infty.
\end{align}
Consequently, we obtain  $\norm{\1_A(\cdot)f^\ast(\cdot)v^\ast}_{\infty} \le m\norm{v^\ast}$ which implies the continuity of $Z_A \colon V^\ast\to L_P^0(\Omega;\R)$ by the continuity of $J\colon  \G([0,T];U^\ast)\to L_P^0(\Omega;\R)$.
The form of the characteristic function follows immediately from \eqref{eq.charfuncJ}.
\end{proof}

The cylindrical random variable $Z_A$ is called the {\em cylindrical integral of $f$ on $A$}. However, we want to define a genuine $V$-valued random variable as the stochastic integral of $f$ which we achieve in the following way:
\begin{definition}\label{de.stochint}
  A function $f\colon [0,T]\to \L(U,V)$ is called {\em stochastically integrable  w.r.t.\ $L$} if $f$ is weakly in $\G([0,T];U^\ast)$ and if for each $A\in \Borel([0,T])$
  there exists an $I_A \in L^0_P(\Omega; V)$ such that
  \begin{align}\label{eq.Z-induced-I}
    \scaprob{I_A}{v^\ast}=Z_A v^\ast \qquad\text{  for all }v^\ast\in V^\ast,
  \end{align}
where $Z_A$ denotes the cylindrical integral of $f$ on $A$. In this case $I_A$ is denoted by
\begin{align*}
 \int_A f(s)\,dL(s):=I_A.
\end{align*}
\end{definition}

The existence of a random variable $I_A\in L^0_P(\Omega; V)$ satisfying \eqref{eq.Z-induced-I}
is called that {\em $Z_A$ is induced by $I_A$}. Such a random variable $I_A$ exists if and only if the cylindrical distribution of $Z_A$ extends to a measure,  see \cite[Th.IV.2.5]{Vaketal}. Our approach by the cylindrical integral enables us to give sufficient and necessary conditions for the extension of the cylindrical distribution of $Z_A$ to a measure. In an arbitrary Banach space, as in the next Theorem, these conditions are rather abstract. However, in more specific spaces, such as Hilbert spaces or Banach spaces of type $p\in [1,2]$, some or even all of these conditions can be
simplified significantly. We demonstrate this by a few subsequent corollaries.

\begin{theorem}\label{th.integrability}
Let  $f\colon [0,T]\to \L(U,V)$ be a function which is weakly in $\G([0,T];U^\ast)$. Then $f$ is stochastically integrable if and only if the following is satisfied:
\begin{enumerate}
\item[{\rm (1)}] the mapping $T_a$ is weak$^\ast$-weakly sequentially continuous where
\begin{align}\label{cond.drift}
  T_a\colon V^\ast\to L^1([0,T];\R),\qquad T_a v^\ast=a(f^\ast(\cdot)v^\ast);
\end{align}
\item[{\rm (2)}] there exists a Gaussian covariance operator $R\colon V^\ast\to V$ satisfying
\begin{align}\label{cond.Covariance-int}
     \scapro{v^\ast}{Rv^\ast}= \int_0^T \scapro{f^\ast(s)v^\ast}{Qf^\ast(s)v^\ast}\,ds
    \qquad\text{for all }v^\ast\in V^\ast;
\end{align}
\item[{\rm (3)}]  for $\chi\colon [0,T]\times U\to V$  defined by $\chi(s,u):=f(s)u$,  the cylindrical measure
\begin{align}\label{cond.Levy-int}
   \nu \colon \Z(V)\to [0,\infty], \qquad \nu = (\mu\otimes \leb)\circ \chi^{-1},
\end{align}
extends to a measure and L{\'e}vy measure on $\Borel(V)$.
\end{enumerate}
\end{theorem}

Note that although stochastic integrability of $f$ requires that the cylindrical integral $Z_A$ is induced by a random variable $I_A$ for all $A\in \Borel([0,T])$, Conditions \eqref{cond.Covariance-int} and \eqref{cond.Levy-int} can be considered as conditions only for $A=[0,T]$. Condition \eqref{cond.drift} is the best sufficient and necessary prescription for the first term of the characteristics available, in order to guarantee integrability for the following reason: if $L$ is a genuine L{\'e}vy process with classical characteristics $(b,0,0)$ for some $b\in U$, then stochastic integrability of the function $f$ according to Definition \ref{de.stochint}  reduces to Pettis integrability of the function $f(\cdot)b$. In this case,
$a(\cdot)=\scapro{b}{\cdot}$ and Condition \eqref{cond.drift} is known to be equivalent to Pettis integrability of $f(\cdot)b$, see \cite[Th.4.1]{Musial}. In Lemma \ref{le.drift-cond-Pettis} we explain this equivalence more general for genuine L{\'e}vy processes with arbitrary characteristics $(b,Q,\mu)$.

In light of Lemma \ref{le.Zdistribution}, Conditions \eqref{cond.Covariance-int} and \eqref{cond.Levy-int} seem to be a straightforward conclusion, but the change from a cylindrical to a genuine infinitely divisible measure requires some further arguments. The correction term between the classical L{\'e}vy-Khintchine formula in \eqref{eq.charLevy} and the cylindrical version in
\eqref{eq.charcylLevy} is considered in the following lemma.
\begin{lemma}\label{le.correction-term}
  If $\xi$ is a L{\'e}vy measure on $\Borel(V)$ then the function
\begin{align*}
\Delta_\xi\colon V^\ast\to \R, \qquad \Delta_\xi(v^\ast):= \int_V
\scapro{v}{v^\ast}
  \big(\1_{B_{V}}(v)-\1_{B_{\R}}(\scapro{v}{v^\ast})\big)\, \xi(dv),
\end{align*}
is well defined and satisfies $\Delta_\xi(v_n^\ast)\to 0$ for a sequence
$(v_n^\ast)_{n\in\N}\subseteq V^\ast$ converging weakly$^\ast$ to $0$.
\end{lemma}
\begin{proof}
Let $v^\ast\in V^\ast$ and define $D(v^\ast):=\{v\in V:\, \abs{\scapro{v}{v^\ast}}\le 1\}$. It follows for each $v\in V$  that
\begin{align}\label{eq.correction}
 \abs{\scapro{v}{v^\ast}}\abs{\1_{B_V}(v)-\1_{B_{\R}}(\scapro{v}{v^\ast})}
&=  \abs{\scapro{v}{v^\ast}}\Big(\1_{B_V\cap (D(v^\ast))^c}(v) + \1_{B_V^c \cap D(v^\ast)}(v)  \Big)\nonumber \\
&\le \abs{\scapro{v}{v^\ast}}^2 \1_{B_V}(v) + \1_{B_V^c}(v).
\end{align}
Proposition 5.4.5 in \cite{Linde} gurantees
\begin{align}\label{eq.estimate-correction}
\begin{split} \int_{V}\abs{\scapro{v}{v^\ast}}\abs{\1_{B_V}(v)-\1_{B_{\R}}(\scapro{v}{v^\ast})}\,\xi(dv)
\le \int_{B_V}  \scapro{v}{v^\ast}^2\,\xi(dv)+\xi(B_V^c)<\infty,
\end{split}
\end{align}
which shows that the function $\Delta_\xi$ is well defined.
Let $(v_n^\ast)_{n\in\N}\subseteq V^\ast$ be a sequence converging weakly$^\ast$ to $0$. An analogue estimate as in \eqref{eq.correction} shows for all $n\in\N$:
\begin{align}\label{eq.correction-weakseq}
\begin{split}
&\int_V\abs{\scapro{v}{v^\ast_n}}\abs{\1_{B_V}(v)-\1_{B_{\R}}(\scapro{v}{v^\ast_n})}
\,\xi(dv) \\
&\qquad\qquad \le \int_{B_V} \abs{\scapro{v}{v^\ast_n}}^2 \,\xi(dv) +\int_{B_V^c}  \1_{D(v_n^\ast)}(v)\abs{\scapro{v}{v^\ast_n}}\,\xi(dv).
\end{split}
\end{align}
For $\alpha:=\sup_{n\in\N}\norm{v_n^\ast}$ define the mapping $m_\alpha\colon V\to V$ by
$m_\alpha(v):=\alpha^{-1}v$. Then $\tilde{\xi}_\alpha:=(\xi+\xi^{-})\circ m_\alpha^{-1}$ is a L{\'e}vy measure on $\Borel(V)$, and thus there exists an infinitely divisible  measure $\theta$ on $\Borel(V)$ with
characteristics $(0,0,\tilde{\xi}_\alpha)$.  The inequality $1-\cos(\beta)\ge \tfrac{1}{3}\beta^2$ for all $\abs{\beta}\le 1$ implies by using the symmetry of $\tilde{\xi}_\alpha$  that the characteristic function of $\theta$ satisfies for each $n\in\N$:
\begin{align*}
  \phi_{\theta}(v^\ast_n)
  &= \exp\left(-\int_V \big(1- \cos(\scapro{v}{v^\ast_n})\big)\, (\xi+\xi^-)\circ m_\alpha^{-1}(dv)\right)\\
& \le  \exp\left(-\int_{B_V}\big( 1- \cos(\alpha^{-1}\scapro{v}{v^\ast_n})\big)\, (\xi+\xi^-)(dv)\right)\\
& \le  \exp\left(-\tfrac{2}{3\alpha^2}\int_{B_V} \abs{\scapro{v}{v^\ast_n}}^2\, \xi(dv)\right).
\end{align*}
Since the characteristic function $\phi_\theta$ is weakly$^\ast$ sequentially continuous we obtain
\begin{align*}
  \int_{B_V} \abs{\scapro{v}{v^\ast_n}}^2 \xi(dv)
  \le -\tfrac{3\alpha^2}{2}\ln \left(\phi_{\theta}(v^\ast_n)\right)
  \to 0 \qquad\text{as $n\to\infty$.}
\end{align*}
Since $\xi(B_V^c)<\infty$ and $\1_{D(v_n^\ast)}(v)\abs{\scapro{v}{v^\ast_n}}\le
 1$ for all $n\in\N$ and $v\in B_V^c$, Lebesgue's theorem of dominated convergence implies
\begin{align*}
  \int_{B_V^c}  \1_{D(v_n^\ast)}(v)\abs{\scapro{v}{v^\ast_n}}\,\xi(dv)
  \to 0  \qquad\text{as $n\to\infty$,}
\end{align*}
which completes the proof by  \eqref{eq.correction-weakseq}.
\end{proof}

\begin{lemma}\label{le.drift-cond-Pettis}
Let   $f\colon [0,T]\to \L(U,V)$ be a function which is weakly in $\G([0,T];U^\ast)$.
Then  Condition \eqref{cond.drift} is satisfied if and only if
\begin{enumerate}
\item[{\rm (1)}]  for every sequence  $(v_n^\ast)_{n\in\N}\subseteq V^\ast$ converging weakly$^\ast$ to $0$ and  $A\in \Borel([0,T])$ we have
\begin{align}\label{cond.drift-2}
  \lim_{n\to\infty} \int_A a(f^\ast(s)v^\ast_n)\,ds =0.
\end{align}
\end{enumerate}
If $L$ is a genuine L{\'e}vy process with characteristics $(b,Q,\mu)$ then its cylindrical characteristics $(a,Q,\mu)$ satisfies $a=\scapro{b}{\cdot}-\Delta_\mu$ and (1) is equivalent to
\begin{enumerate}
  \item[{\rm (2)}] \inlineequation[cond.drift-Pettis]{\text{the mapping $t\mapsto f(t)b$ is Pettis integrable.}\hfill}
\end{enumerate}
\end{lemma}
\begin{proof}
  Since $a\colon U^\ast\to \R$ maps bounded sets to bounded sets according to \eqref{eq.drift-bounded-bounded},
Condition \eqref{cond.drift-2} implies \eqref{cond.drift} by standard arguments.

For the second part assume that $L$ is a genuine L{\'e}vy process and we adopt the notation $\Delta_\mu$ from Lemma \ref{le.correction-term} but for the L{\'e}vy measure $\mu$ on $\Borel(U)$.
The first entries of the classical characteristics $(b,Q,\mu)$ and of the cylindrical characteristics $(a,Q,\mu)$ obey
$\scapro{b}{u^\ast}=a(u^\ast)+\Delta_\mu(u^\ast)$ for all $u^\ast\in U^\ast$,
which implies for all $v^\ast\in V^\ast$ and $h\in L^{\infty}([0,T];\R)$ the identity
\begin{align}\label{eq.banda}
    \int_0^T h(s) \scapro{f(s)b}{v^\ast}\,ds
  =\int_0^T h(s)\, a\big(f^\ast(s) v^\ast\big)\,ds + \int_0^T h(s)\, \Delta_\mu\big(f^\ast(s)v^\ast\big)\,ds .
\end{align}
Let $m:=\sup_{s\in [0,T]}\norm{f^\ast(s)}_{V^\ast\to U^\ast}$ be defined as in \eqref{eq.funiformbounded} and let $(v_n^\ast)_{n\in\N}\subseteq V^\ast$ be a sequence weakly$^\ast$ converging to $0$ with setting $\alpha:=\sup_{n\in\N}\norm{v_n^\ast}$. From \eqref{eq.estimate-correction} and \cite[Pro.5.4.5]{Linde} we conclude
\begin{align*}
  \sup_{n\in\N}\sup_{s\in [0,T]} \abs{\Delta_\mu(f^\ast(s)v_n^\ast)}&\le
 \sup_{\norm{u^\ast}\le \alpha m} \int_{B_U} \scapro{u}{u^\ast}^2\, \mu(du)+\mu(B_U^c)<\infty.
\end{align*}
By applying Lebesgue's theorem of dominated convergence and Lemma \ref{le.correction-term}, we obtain for every $h\in L^{\infty}([0,T];\R)$ that
\begin{align}\label{eq.Delta-0}
 \abs{ \int_0^T h(s)\Delta_\mu(f^\ast(s)v_n^\ast)\,ds}
 \le \norm{h}_\infty   \int_0^T \abs{\Delta_\mu(f^\ast(s)v_n^\ast)}\,ds\to 0.
\end{align}
Let $S_b\colon V^\ast\to L^1([0,T];\R)$ denote the mapping defined by $S_b v^\ast=\scapro{f(s)b}{v^\ast}$. It follows from \eqref{eq.banda} and \eqref{eq.Delta-0}
that $S_b$ is weak$^\ast$-weakly sequentially continuous if and only if the mapping $T_a$, defined in \eqref{cond.drift}, is weak$^\ast$-weakly sequentially continuous. Since the mapping
$t\mapsto f(t)b$ is Pettis integrable if and only if $S_b$ is weak$^\ast$-weakly  continuous according to \cite[Th.4.1]{Musial} and since $V$ is separable, the proof is completed.
\end{proof}

\begin{proof} (Theorem \ref{th.integrability}).  Sufficiency:
according to \eqref{eq.L-decompose} the cylindrical L{\'e}vy process $L$ can be decomposed into $L(t)=W(t)+P(t)$ for all $t\ge 0$, where $W$ and $P$ are independent, cylindrical L{\'e}vy processes with characteristics $(0,Q,0)$ and $(a,0,\mu)$,
respectively. For the cylindrical integral $Z_A$ of $f$ on $A\in\Borel([0,T])$ we obtain \begin{align*}
  Z_Av^\ast=Z^W_A v^\ast+ Z^P_A v^\ast \qquad\text{for all }v^\ast\in V^\ast,
\end{align*}
where $Z^W_A$ is the cylindrical integral w.r.t.\ $W$ on $A$ and $Z^P_A$ w.r.t.\ $P$ on $A$.

Lemma \ref{le.Zdistribution} implies that the cylindrical random variable $Z^W_A$ is Gaussian with
covariance $r_A$ defined in \eqref{eq.3charZ_f.2}. Since the covariance satisfies
\begin{align*}
 r_A(v^\ast)
   \le  \int_0^T \scapro{f^\ast(s)v^\ast}{Qf^\ast(s)v^\ast}\, ds= \scapro{v^\ast}{Rv^\ast} \qquad\text{for all }v^\ast\in V^\ast,
\end{align*}
Theorem \ref{th.Gauss-dominated} implies that the cylindrical distribution of $Z^W_A$ extends to a measure on $\Borel(V)$. It follows from Theorem IV.2.5 in  \cite{Vaketal} that there exists a random variable $I^W_A\in L^0_P(\Omega;V)$ with $\scapro{I^W_A}{v^\ast}=Z_A^Wv^\ast$ for all $v^\ast\in V^\ast$.

According  to Lemma \ref{le.Zdistribution} the cylindrical random variable $Z_A^P$ has
the L{\'e}vy measure $\nu_A$, which  obeys for every set $C\in \Z(V)$ the inequality
\begin{align*}
  \nu_{A}(C)=\int_A \int_U \1_C(f(s)u)\, \nu(du)\,ds
  \le \int_0^T \int_U \1_C(f(s)u)\, \nu(du)\,ds= \nu(C).
\end{align*}
Theorem \ref{th.Levy-dominated} implies that $\nu_{A}$ extends to a L{\'e}vy measure on $\Borel(V)$,
which is also denoted by $\nu_A$. Thus, there exists a probability measure $\theta_{A}$ on $\Borel(V)$ with characteristic function
\begin{align}\label{eq.char-rho}
  \phi_{\theta_A}(v^\ast)
  &=\exp\left( \int_V \left(e^{i\scapro{v}{v^\ast}}-1-i\scapro{v}{v^\ast}\1_{B_V}(v)\right)\, \nu_{A}  (dv)\right)
   \qquad\text{for all }v^\ast\in V^\ast.
\end{align}
Define the function
\begin{align*}
  c_A\colon V^\ast\to \R,\qquad
  c_A(v^\ast):= \int_A a(f^\ast(s)v^\ast)\,ds + \Delta_{\nu_A}(v^\ast),
\end{align*}
where the function $\Delta_{\nu_A}$ is defined in Lemma \ref{le.correction-term}.
Let $X_A$ be a $V$-valued random variable with probability distribution $\theta_A$.
 It follows from Lemma \ref{le.Zdistribution}  and \eqref{eq.char-rho} that
 \begin{align}\label{eq.Z_A=dirac-conv-X_A}
   Z_A^Pv^\ast\stackrel{d}{=} c_A(v^\ast)+ \scapro{X_A}{v^\ast} \qquad\text{for all }v^\ast\in V^\ast,
 \end{align}
where $\stackrel{d}{=}$ denotes equality in distribution. Linearity and continuity of $Z_A^P$ and $\scapro{X_A}{\cdot}$ implies that $c_A\in V^{\ast\ast}$.  It follows from Condition \eqref{cond.drift}, Lemma \ref{le.correction-term} and Lemma \ref{le.drift-cond-Pettis} that $c_A$ is weakly$^\ast$ sequentially continuous.
Since $V$ is separable, we obtain that $c_A$ is weakly$^\ast$ continuous (\cite[Co.2.7.10]{Megginson}), which implies $c_A\in V$. Consequently, we can define the Dirac measure $\delta_{c_A}$ at the point $c_A\in V$. It follows from \eqref{eq.Z_A=dirac-conv-X_A} that the cylindrical distribution of $Z_A^P$ extends to the convolution $\delta_{c_A}\ast \theta_A$.
Consequently,  \cite[Th.IV.2.5]{Vaketal} guarantees that there exists a random variable $I_A^P\in L^0_P(\Omega;V)$ with $\scapro{I_A^P}{v^\ast}=Z_A^Pv^\ast$ for all
$v^\ast\in V^\ast$, which shows the stochastic integrability of $f$.

Necessity: let $I_A\in L^0_P(\Omega;V)$ denote the stochastic integral of $f$ on $A$ and let $(c_A,S_A, \xi_A)$ be the characteristics of the
infinitely divisible random variable $I_A$. Then, due to the uniqueness of the L{\'e}vy-Khintchine formula in $\R$, for each $v^\ast\in V^\ast$ the characteristics of the real valued random variables $\scapro{I_A}{v^\ast}$ and $Z_Av^\ast$ coincide which results in
\begin{align}
\scapro{c_A}{v^\ast}-\Delta_{\xi_A}(v^\ast)&=
 \int_A a\left( f^\ast(s)v^\ast\right)\, ds,\label{eq.same0}\\
\scapro{v^\ast}{S_A v^\ast}&= \int_A \scapro{f^\ast(s)v^\ast}{Qf^\ast(s)v^\ast}\,ds,\label{eq.same1}\\
\xi_A \circ (v^\ast)^{-1}
&= \Big(\big(\mu\otimes \leb\big)\circ \chi_A^{-1}\Big)\circ (v^\ast)^{-1}.\label{eq.same2}
\end{align}
Here, we obtain the characteristics of $\scapro{I_A}{v^\ast}$ on the left hand side by a standard calculation for the transform of an infinitely divisible measure under a linear mapping (see e.g.\ \cite[Pro.11.10]{Sato}), whereas the characteristics of $Z_Av^\ast$ on the right hand side is given in Lemma \ref{le.Zdistribution}.  Equation \eqref{eq.same0} shows Condition \eqref{cond.drift} due to Lemma \ref{le.correction-term} and Lemma \ref{le.drift-cond-Pettis}. By choosing $A:=[0,T]$, identity  \eqref{eq.same1} implies Condition \eqref{cond.Covariance-int}.

By using \eqref{eq.same2} and  \cite[Pro.11.10]{Sato}, it follows from linearity that the
$\R^n$-valued infinitely divisible random variables $\big(\scapro{I_A}{v_1^\ast},\dots, \scapro{I_A}{v_n^\ast}\big)$ and $\big(Z_Av^\ast_1,\dots, Z_Av^\ast_n)$ have the same L{\'e}vy measures for all $v_1^\ast,\dots, v_n^\ast \in V^\ast$ and $n\in\N$, i.e.\
\begin{align*}
  \xi_A\circ \pi_{v_1^\ast,\dots, v_n^\ast}^{-1}= \Big(\big(\mu\otimes \leb\big)\circ \chi_A^{-1}\Big) \circ \pi_{v_1^\ast,\dots,v_n^\ast}^{-1}.
\end{align*}
Consequently, by choosing $A=[0,T]$ we have $\nu=(\mu\otimes \text{leb})\circ \chi_A^{-1}$ and the  image cylindrical measure $\nu$ extends to the L{\'e}vy measure $\xi_A$.
\end{proof}

\begin{remark}
  In the work \cite{RiedleGaans} together with van Gaans, we developed a stochastic integral for deterministic integrands w.r.t.\ martingale valued measures in Banach spaces, i.e.\ in particular with respect to the compensated Poisson random measure of a classical L{\'e}vy process in a Banach space $U$. The integrability of a function is described in terms of the appropriate convergence of a random series, which is very similar to the case of $\gamma$-radonifying operators. This approach cannot be applied to cylindrical L{\'e}vy processes as they do not satisfy an It{\^o}-L{\'e}vy decomposition in the underlying Banach space but only in $\R$, which is then depending on the argument in $U^\ast$. Both integrals in the current work and in \cite{RiedleGaans} underly a kind of a stochastic version of the Pettis integral.
\end{remark}

As mentioned before, we simplify the conditions in Theorem \ref{th.integrability} in some more
specific spaces. We begin with the most important case of a Hilbert space.
\begin{theorem}\label{th.intHilbert}
Assume that $V$ is a Hilbert space with orthonormal basis $(e_k)_{k\in\N}$ and let  $f\colon [0,T]\to \L(U,V)$ be a function which is weakly in $\G([0,T];U^\ast)$. Then $f$ is stochastically integrable if and only if the following is satisfied:
\begin{enumerate}
\item[{\rm (1)}] the  mapping $T_a$ is weak-weakly sequentially continuous where
\begin{align}\label{cond.Hilbert-drift}
  T_a\colon V^\ast\to L^1([0,T];\R),\qquad T_a v^\ast=a(f^\ast(\cdot)v^\ast);
\end{align}
\item[{\rm (2)}] \inlineequation[cond.Hilbert-Covariance-int]{\displaystyle \int_0^T \text{\rm tr}\left[f(s)Qf^\ast (s) \right]\,ds <\infty;\hfill}
\item[{\rm (3)}]\inlineequation[cond.Hilbert-Levy-int]{\displaystyle    \limsup_{m\to\infty}\sup_{n\ge m}
   \int_0^T \int_U \left(\sum_{k=m}^n \scapro{u}{f^\ast(s)e_k}^2 \wedge 1\right)\, \mu(du)\,ds=0.\hfill}
\end{enumerate}

\end{theorem}
\begin{proof}
The closed graph theorem shows that
\begin{align*}
  \scapro{v^\ast}{Rw^\ast}=\int_0^T \scapro{f^\ast(s)v^\ast}{Qf^\ast(s)w^\ast}\,ds
  \qquad\text{for all }v^\ast, w^\ast\in V^\ast,
\end{align*}
defines a positive, symmetric and bounded operator $R\colon V^\ast\to V$. By applying Tonelli's theorem we can conclude that the operator $R$ is of trace class if and only if Condition \eqref{cond.Hilbert-Covariance-int} is satisfied. Since the space of Gaussian covariance operators in Hilbert spaces coincide with the space of trace class operators by \cite[Th.IV.2.4]{Vaketal}, we have established the equivalence of Conditions \eqref{cond.Covariance-int} and \eqref{cond.Hilbert-Covariance-int}.

If $f$ is stochastically integrable then Theorem \ref{th.integrability} implies
that the cylindrical measure $\nu$, defined in \eqref{cond.Levy-int}, extends to a  measure and it is a L{\'e}vy measure in $V$. Since $V$ is a Hilbert space, the latter implies (see \cite[Th.VI.4.10]{Para}), that
\begin{align*}
\int_V \left( \norm{v}^2 \wedge 1\right) \, \nu(dv)<\infty,
\end{align*}
which shows Condition \eqref{cond.Hilbert-Levy-int}.

It remains to show that \eqref{cond.Hilbert-Levy-int} implies Condition \eqref{cond.Levy-int}, for which we can assume that the cylindrical characteristics of $L$ is of the form $(a,0,\mu)$.
We define the space $V_n:=\text{span}\{e_1,\dots, e_n\}$ and we denote by $\pi_n\colon V\to V$ the orthogonal projection on $V_n$ for each $n\in\N$. Let  $Z_A$ denote the cylindrical integral of $f$ on $A\in \Borel([0,T])$, which has the  characteristics $(b_A, 0, \nu_A)$ according to  Lemma \ref{le.Zdistribution}. If $Z_A^\prime $ denotes an independent copy of $Z_A$ then $\tilde{Z}_A:=Z_A + Z_A^\prime$ is a cylindrical random variable with characteristics $(0,0, \nu_A+ \nu_A^{-})$.
Since $\pi_n$ is a Hilbert-Schmidt operator, the cylindrical distribution of
$\tilde{Z}_A\circ \pi_n^\ast $ extends to a probability  measure $\theta_n$ on  $\Borel(V)$ due to \cite[Th.VI.5.2]{Vaketal}, which is infinitely divisible with characteristics $(0,0,\xi_n)$ where $\xi_n:=(\nu_A+\nu_A^{-})\circ \pi_n^{-1}$.
By using the inequality $1-\cos\beta\le 2(\beta^2\wedge 1)$ for all $\beta\in\R$ we obtain for every $v^\ast\in V $ that
\begin{align*}
 1- \phi_{\theta_n}(v^\ast)&=
  1-\exp\left( \int_{V} \big(\cos (\scapro{v}{v^\ast})-1\big)\, \xi_n(dv)\right)\\
 &\le \int_V \big(1-\cos(\scapro{v}{v^\ast})\big)\, \xi_n(dv)
 \le 2 \int_V \left(\scapro{v}{v^\ast}^2\wedge 1 \right)\, \xi_n(dv).
\end{align*}
Let $g_m$ denote the density of the standard normal distribution on $\Borel(\R^m)$.
For every $m,n\in\N$ with $m\le n$ it follows that
\begin{align*}
&\int_{\R^{n-m+1}} \Big(1-\text{Re}\,\phi_{\theta_n}(\beta_m e_m+\dots +\beta_ne_n)\Big)
  g_{n-m+1}(\beta_m,\dots, \beta_n)\, d\beta_m\cdots d\beta_n\\
&\qquad \le 2\int_{\R^{n-m+1}}\int_V \left(\abs{\sum_{k=m}^n \beta_k\scapro{v}{e_k}}^2\wedge 1\right) \,
   \xi_n(dv)\,  g_{n-m+1}(\beta_m,\dots, \beta_n)\, d\beta_m\cdots d\beta_n  \\
&\qquad \le 2\int_V  \left( \left(\int_{R^{n-m+1}}\left[\abs{\sum_{k=m}^n \beta_k\scapro{v}{e_k}}^2\right]g_{n-m+1}(\beta_m,\dots, \beta_n)\, d\beta_m\cdots d\beta_n \right) \wedge 1 \right) \,
    \xi_n(dv) \\
&\qquad =2\int_V \left( \sum_{k=m}^n\scapro{v}{e_k}^2\wedge 1\right)
    \,\xi_n(dv)\\
&\qquad =2\int_V \left( \sum_{k=m}^n\scapro{\pi_n v}{e_k}^2\wedge 1\right)
    \,(v_A+v_A^{-1})(dv)\\
&\qquad = 4\int_V \left( \sum_{k=m}^n\scapro{ v}{e_k}^2\wedge 1\right)
    \,v(dv).
\end{align*}
Condition \eqref{cond.Hilbert-Levy-int} implies
\begin{align*}
\limsup_{m\to\infty}\sup_{n\ge m}\int_{\R^{n-m+1}} \!\! \big(1-\text{Re}\,\phi_{\theta_n}(\beta_m e_m+\dots +\beta_ne_n)\big)
  g_{n-m+1}(\beta_m,\dots, \beta_n)\, d\beta_m\cdots d\beta_n=0,
\end{align*}
which shows  by \cite[Le.VI.2.3]{Para} that the family $\{\theta_n\}_{n\in\N}$ is relatively compact in $M(V)$. Since $\tilde{Z}_A$ and thus its characteristic function $\phi_{\tilde{Z}_A}$ are continuous (see \cite[Pro.IV.3.4]{Vaketal}), we have for each $v^\ast\in V^\ast$:
\begin{align}\label{eq.conchar}
  \lim_{n\to\infty} \phi_{\theta_n}(v^\ast)
  = \lim_{n\to\infty}\phi_{\tilde{Z}_A}\big(\scapro{e_1}{v^\ast}e_1+\dots +\scapro{e_n}{v^\ast}e_n\big)
  = \phi_{\tilde{Z}_A}(v^\ast).
\end{align}
Together with the relative compactness of $\{\theta_n\}_{n\in\N}$ it follows from Theorem IV.3.1 in \cite{Vaketal} that the probability measures $\{\theta_n\}_{n\in\N}$ converges weakly to a measure $\theta$, which coincides with the cylindrical distribution of $\tilde{Z}_A$ on $\Z(V)$. Consequently,  \cite[Th.IV.2.5]{Vaketal} guarantees that there exists a random variable $\tilde{I}_A \in L^0_P(\Omega;V)$ with $\scapro{\tilde{I}_A}{v^\ast}=\tilde{Z}_Av^\ast$ for all
$v^\ast\in V^\ast$. Thus, the cylindrical L{\'e}vy measure $\nu_A+\nu_A^{-}$ of $\tilde{Z}_A$ extends to the L{\'e}vy measure of $\tilde{I}_A$. Since
\begin{align*}
  \nu_A(C)\le \nu_A(C)+\nu_A^{-}(C) \qquad
  \text{for all }C\in \Z(V),
\end{align*}
Theorem \ref{th.Levy-dominated} implies that $\nu_A$ extends to a L{\'e}vy measure which shows Condition \eqref{cond.Levy-int}.
\end{proof}

\begin{remark}\label{re.anna}
  In the work \cite{Anna} on L{\'e}vy processes in Hilbert spaces, the author developes among others  a theory of stochastic integration for deterministic operators
  with respect to a classical L{\'e}vy process in a separable Hilbert space $U$. More specifically, let $V$ be another separable Hilbert space and define the set
  \begin{align*}
    S:=\left\{f\colon [0,T]\to \L(U,V):\, f\text{ is strongly measurable  and }
    \int_0^T \norm{f(s)}_{U\to V}^2\,ds <\infty\right\}.
  \end{align*}
Then by using tighness conditions for infinitely divisible measures in Hilbert spaces (see \cite{Para}), a stochastic integral is defined for integrands in $S$ in \cite{Anna}. In this case of genuine L{\'e}vy processes,  it is easy to see that each function $f\in S$ satisfies Condition \eqref{cond.drift-Pettis} in Lemma \ref{le.drift-cond-Pettis} and Conditions \eqref{cond.Hilbert-Covariance-int} and \eqref{cond.Hilbert-Levy-int} in Theorem \ref{th.intHilbert}. Thus, Theorem \ref{th.intHilbert} guarantees that each $f\in S$ is stochastically integrable in our sense according to Definition \ref{de.stochint}.
\end{remark}

\begin{corollary}\label{co.stablecotype}
Under the assumption of Theorem \ref{th.integrability}  let the cylindrical measure $\nu$ be defined by \eqref{cond.Levy-int}. Then we have the following:
\begin{enumerate}
  \item[{\rm (a)}] If $V$ is of type $p\in [1,2]$, then Condition \eqref{cond.Levy-int} replaced by
     \begin{enumerate}
       \item[{\rm (3$^\prime$)}] $\nu$ extends to a measure and $\int_V \Big(\norm{v}^p\wedge 1 \Big)\,\nu(dv)<\infty$,
     \end{enumerate}
     implies together with \eqref{cond.drift} and \eqref{cond.Covariance-int}
     that $f$ is stochastically integrable.
  \item[{\rm (b)}] If  $\;V$ is of cotype $q\in [2,\infty)$ then  stochastic integrability of $f$ implies
     \begin{enumerate}
       \item[{\rm (3$^\prime$)}] $\nu$ extends to a measure and $\int_V \Big(\norm{v}^q\wedge 1 \Big)\,\nu(dv)<\infty$.
     \end{enumerate}
\end{enumerate}
\end{corollary}
\begin{proof}
The reformulation of Condition \eqref{cond.Levy-int} in Theorem \ref{th.integrability} follows in both cases from results in the article \cite{AraujoGine2}. In this work the authors establish sufficient conditions in Banach spaces of type $p$ and necesssary conditions in  Banach spaces of cotype $q$ for a $\sigma$-finite measure to be a L{\'e}vy measure.
\end{proof}

\begin{corollary}
If $V$ is the space  $\ell^p(\R)$ of sequences for  $p\in [2,\infty)$ equipped with the standard basis $(e_k)_{k\in\N}$,  then Condition \eqref{cond.Covariance-int} and \eqref{cond.Levy-int} in Theorem \ref{th.integrability} can be replaced by
  \begin{enumerate}
\item[{\rm (2$^\prime$)}] \inlineequation[cond.sequence-covariance]{\displaystyle \sum_{k=1}^\infty \left(\int_0^T \scapro{f^\ast(s)e_k}{Qf^\ast(s)e_k}\,ds  \right)^{p/2} <\infty. \hfill }
\item[{\rm (3$^\prime$)}] the cylindrical measure $\nu$  extends to a Radon measure on $\Borel(V)$ and satisfies
\begin{enumerate}
 \item[{\rm (a)}] \inlineequation[eq.sequence-Levy1]{\displaystyle\int_V \left( \norm{v}^{p} \wedge 1\right) \,\nu(dv)<\infty;\hfill}
 \item[{\rm (b)}] \inlineequation[eq.sequence-Levy2]{\displaystyle\sum_{k=1}^\infty \left(\int_{\norm{v}\le 1} \abs{\scapro{v}{e_k}}^2  \,\nu(dv)\right)^{p/2} <\infty.\hfill}
\end{enumerate}
\end{enumerate}
  \end{corollary}
\begin{proof}
  Theorem 5.6 in \cite{Vaketal} guarantees that Conditions \eqref{cond.Covariance-int} and \eqref{cond.sequence-covariance} are equivalent. The class of L{\'e}vy measures in $\ell^p(\R)$ is described by Conditions \eqref{eq.sequence-Levy1}  and \eqref{eq.sequence-Levy2} according to a result in
\cite{Jurinskii}.
\end{proof}

\section{Ornstein-Uhlenbeck processes}

In this last part we apply the previous developed theory of stochastic integration to define Ornstein-Uhlenbeck processes driven by cylindrical L{\'e}vy processes. These processes are important since they are solutions of stochastic evolution equations driven by cylindrical L{\'e}vy processes, see for instance \cite{PeszatZab}. We do not study this connection in this work but we consider examples of specific cases, then these processes exist.

If $L$ is a cylindrical L{\'e}vy process in a Banach space $U$ and $G\in \L(U,V)$
then the cylindrical L{\'e}vy process $GL$ defined by $(GL)(t)v^\ast:=L(t)(G^\ast v^\ast)$ for all $v^\ast\in V^\ast$ and $t\ge 0$ is a cylindrical L{\'e}vy process in $V$. It follows for a function $f\colon [0,T]\to \L(U,V)$ that if $f(\cdot)\circ G$ is stochastically integrable w.r.t. $L$ then $f$ is stochastically integrable w.r.t $GL$ and
\begin{align*}
  \int_0^T f(s)G\,dL(s)= \int_0^T f(s)\, d(GL)(s).
  \end{align*}
Thus, without loss of generality we can assume in this section $U=V$, i.e. $L$ is a cylindrical L{\'e}vy process in the separable Banach space $V$.

\begin{definition}\label{de.OU}
If a strongly continuous semigroup $(T(t))_{t\in [0,T]}$ on $V$ is stochastically integrable with respect to $L$, then we call the stochastic process $(X(t):\, t\in [0,T]$) defined by
\begin{align*}
X(t):=T(t)v_0+ \int_0^t T(t-s)\,dL(s)  \qquad\text{for all }t\in [0,T],
\end{align*}
{\em Ornstein-Uhlenbeck process with initial value $v_0\in V$ driven by $L$}.
\end{definition}

The existence of the stochastic convolution integral in Definition \ref{de.OU} is guaranteed by  the following result.
\begin{lemma}
A function $f\colon [0,T]\to \L(V,V)$ is stochastically integrable if and only if $f(T-\cdot)$ is stochastically integrable. In this case we have
the equality in distribution:
\begin{align*}
  \int_0^T f(s)\,dL(s)\stackrel{d}{=} \int_0^T f(T-s)\,dL(s).
\end{align*}
\end{lemma}
\begin{proof}
If $f$ is  weakly in $\G([0,T];V^\ast)$ then $f(T-\cdot)$ is also  weakly in $\G([0,T];V^\ast)$.
Since Conditions \eqref{cond.drift} - \eqref{cond.Levy-int} in Theorem \ref{th.integrability} are invariant  under a transformation $s\mapsto T-s$ the first part of the Lemma is proved. The identity
\begin{align*}
  \exp\left(\int_0^T \Psi(f^\ast(s)v^\ast)\,ds\right) = \exp\left(\int_0^T \Psi(f^\ast(T-s)v^\ast)\,ds\right) \qquad \text{for all }v^\ast\in V^\ast,
\end{align*}
shows the equality of the distributions by Lemma \ref{le.Zdistribution}
\end{proof}

As an example of an Ornstein-Uhlenbeck process we consider the case of a diagonalisable semigroup  and of a cylindrical L{\'e}vy process defined by a sum acting independently along the eigenbasis of the semigroup, cf.\ Lemma \ref{le.weakconvsum}. This kind of setting is considered in several publications, e.g.\ \cite{PeszatZab12} and \cite{PriolaZabczyk}.
\begin{corollary} \label{le.int-stablenoise}
Assume that $V$ is a Hilbert space and that there exists an orthonormal basis $(e_k)_{k\in\N}$ of $V$ and $(\gamma_k)_{k\in\N}\subseteq\R$ such that the semigroup $(T(t))_{t\in [0,T]}$ satisfies
\begin{align}\label{eq.diag-semigroup}
  T^\ast(t)e_k= e^{\gamma_k t}e_k \qquad\text{for all }t\in [0,T],\, k\in\N.
\end{align}
Let the cylindrical L{\'e}vy process $L$ be of the form
\begin{align*}
  L(t)v^\ast = \sum_{k=1}^\infty \scapro{e_k}{v^\ast}\ell_k(t)\qquad
  \text{for all }v^\ast \in V^\ast,\, t\ge 0,
\end{align*}
where   $(\ell_k)_{k\in\N}$ is a sequence of independent, symmetric L{\'e}vy processes in $\R$
with characteristics $(0,0,\nu_k)$.
Then the semigroup $(T(t))_{t\in [0,T]}$ is stochastically integrable w.r.t.\ $L$ if and only
if
\begin{align}\label{eq.cond-sum-int}
  \sum_{k=1}^\infty \int_0^T \int_{\R} \Big(e^{2\lambda_k s}\abs{\beta}^2\wedge 1\Big)\,\nu_k(d\beta)\,ds<\infty.
\end{align}
\end{corollary}
\begin{proof}
According to Lemma \ref{le.weakconvsum}, the cylindrical L{\'e}vy process $L$ has characteristics $(0,0,\mu)$ satisfying $\big(\mu\circ e_k^{-1}\big)(d\beta) = \nu_k(d\beta)$ for all $k\in\N$.
Independence of the L{\'e}vy processes $(\ell_k)_{k\in\N}$ implies for
the L{\'e}vy measure $\mu\circ \pi_{e_m,\dots, e_n}^{-1}$ of $\big(L(1)(e_m),\dots, L(1)(e_n)\big)$ for $0\le m\le n$
the identity
\begin{align*}
  \mu\circ \pi_{e_m,\dots, e_n}^{-1}
  =\sum_{k=m}^n \underbrace{\delta_0\otimes \cdots \otimes \delta_0}_{k-m \text{ times}}\otimes \; \big( \mu\circ \pi_{e_k}^{-1}\big)\otimes \underbrace{\delta_0\otimes \cdots \otimes\delta_0}_{n-k \text{ times}}.
\end{align*}
Thus, we obtain for all $v\in V$:
\begin{align*}
&\int_0^T \int_V \left(\sum_{k=m}^n \scapro{v}{T^\ast(s)e_k}^2\wedge 1\right)\, \mu(dv)\,ds\\
&\qquad\qquad= \int_0^T \int_V \left(\sum_{k=m}^n \scapro{v}{e^{\lambda_ks} e_k}^2\wedge 1\right)\, \mu(dv)\,ds\\
&\qquad\qquad= \int_0^T \int_{\R^{n-m+1}} \left(\sum_{k=m}^n\abs{e^{\lambda_k s}\beta_k}^2\wedge 1\right)\, \big(\mu\circ \pi_{e_m,\dots,  e_n}^{-1}\big)(d\beta_m \cdots d\beta_n)\,ds\\
&\qquad\qquad= \sum_{k=m}^n \int_0^T \int_{\R} \left(e^{2\lambda_ks}{\beta}^2\wedge 1\right)\, \big(\mu\circ e_k^{-1}\big)(d\beta)\,ds.
\end{align*}
Since $V$ is a Hilbert space, the adjoint semigroup $(T^\ast(t))_{t\in [0,T]}$ is strongly continuous, and thus $(T(t))_{t\in [0,T]}$ is weakly in $\G([0,T];V^\ast)$. An application of Theorem \ref{th.intHilbert} establishes that the semigroup is stochastically integrable if and only if \eqref{eq.cond-sum-int} is satisfied.
\end{proof}

\begin{example}
Assume that the cylindrical L{\'e}vy process $L$ is given as in Example \ref{ex.PriolaZab}
by $\ell_k(\cdot):=\sigma_k h_k(\cdot)$, where $(h_k)_{k\in\N}$ is a sequence of independent,  symmetric, $\alpha$-stable processes and $(\sigma_k)_{k\in\N}\subseteq \R$.  If
the strongly continuous semigroup $(T(t))_{t\in [0,T]}$ satisfies \eqref{eq.diag-semigroup} for
$\lambda_k< 0 $ with $\lambda_k\to -\infty$ for $k\to \infty$, then a simple calculation
shows that \eqref{eq.cond-sum-int} is satisfied if and only if
\begin{align*}
  \sum_{k=1}^\infty \frac{\abs{\sigma_k}^\alpha}{\abs{\gamma_k}}<\infty.
\end{align*}
In this case, the semigroup is stochastically integrable, which coincide with a result in \cite{PeszatZab12}.
\end{example}

We now consider stochastic integrability of a semigroup $(T(t))_{t\ge 0}$ w.r.t.\ to a cylindrical L{\'e}vy noise constructed by subordination as in Example \ref{ex.BrzezniakZabczyk}. In fact, we will show integrability in a possible smaller subspace $E\subseteq V$ with norm $\norm{\cdot}_E$ assuming $T(t)(V)\subseteq E$ for almost all $t>0$. Recall that $H_C$ denotes the reproducing kernel Hilbert space of the subordinated cylindrical Wiener process and $i_C:H_C\to V$ its embedding. In the following denote by $R(H_C,E)$  the space of $\gamma$-radonifying operators $g:H_C\to E$. For $g\in R(H_C,E)$ and $p\in [1,\infty)$
define
\begin{align*}
  \norm{g}_{R_p(H_C,E)}^p:=E\left[\norm{\sum_{k=1}^\infty \gamma_k gh_k}_E^p\right],
\end{align*}
where $(\gamma_k)_{k\in\N}$ is a family of independent, real valued standard normally distributed random variables and $(h_k)_{k\in\N}$ is an orthonormal basis in $H_C$.

\begin{corollary}\label{co.BrzZab}
Let $L$ be the cylindrical L{\'e}vy process in a separable Banach space $V$ defined by
\begin{align}
  L(t)v^\ast:= W\big(\ell(t)\big)v^\ast\qquad\text{for all }v^\ast\in V^\ast,\, t\ge 0,
\end{align}
where $W$ denotes a cylindrical Wiener process with covariance operator $C$ and $\ell$ is a real valued subordinator  with characteristics $(0,0,\rho)$. Let $E$ be a separable Banach space of type $p\in [1,2]$. If the semigroup $(T(t))_{t\in [0,T]}$ in $V$ is weakly in $\G([0,T];V^\ast)$ and the mapping $T(t)\circ i_C$ is in $R(H_C,E)$ for almost all $t\in [0,T]$, then
\begin{align*}
  \int_0^\infty \int_0^T \Big( \abs{r}^{p/2} \norm{T(s)\circ i_C}_{R_p(H_C,E)}^p\wedge 1 \Big)\,ds\,\rho(dr)<\infty,
\end{align*}
implies that the semigroup $(T(t))_{t\in [0,T]}$ is stochastically integrable w.r.t.\ $L$ and the stochastic integral is $E$-valued.
\end{corollary}
\begin{proof}
Let $\gamma$ denote the canonical cylindrical Gaussian measure on $H_C$ and let $\kappa$ be defined as in Lemma \ref{le.BrzZab10}, i.e.\ $\kappa(h,s):=\sqrt{s}i_Ch$, and $\chi$ as in Theorem \ref{th.integrability}, i.e.\ $\chi(s,v):=T(s)v$. For applying Corollary \ref{co.stablecotype} we have to show that the cylindrical measure $\nu=\big(\big((\gamma\otimes \rho)\circ \kappa^{-1}\big)\otimes\, \text{\rm leb}\big)\circ \chi^{-1}$ extends to a measure on $\Borel(V)$.   For this purpose define the family of cylindrical sets
\begin{align*}
  G:&=\big\{ C(v_1^\ast,\dots, v_n^\ast;R):\, v_1^\ast,\dots, v_n^\ast\in V^\odot,\\
  & \qquad \qquad
  R=(a_1,b_1)\times\dots \times (a_n,b_n), \,-\infty\le a_j< b_j\le \infty,\, j=1,\dots, n,\,n\in\N \big\},
\end{align*}
where $V^\odot$ denotes the weak$^\ast$ dense subspace of $V^\ast$ such that $(T^\ast(t))_{t\in [0,T]}$ acts strongly continuously on $V^\odot$. Since $V$ is separable and $V^\odot$
separates points in $V$, Theorem I.2.1 in \cite{Vaketal} guarantees that $G$
generates the $\sigma$-algebra $\Borel(V)$.
Define $\gamma_s:=\gamma\circ \big(T(s)\circ i_C\big)^{-1}$ for all $s\in [0,T]$. Let $\Gamma\colon H_C\to L^0_P(\Omega;\R)$ be a cylindrical random variable with distribution $\gamma$. Since $\Gamma(T^\ast(s_k)v^\ast) \to \Gamma(T^\ast(s)v^\ast)$ in probability for $s_k\to s$ and all $v^\ast\in V^\odot$, the portmanteau theorem in $\R^n$ implies
for each $C:=C(v_1^\ast,\dots, v_n^\ast;R)\in G$:
\begin{align}\label{eq.cont-gamma}
\begin{split}
  \gamma_{s_k}(C)&= P\big( ( \Gamma(T^\ast(s_k)v_1^\ast),\dots,  \Gamma(T^\ast(s_k)v_n^\ast)) \in R\big)\\
  &\to P\big( ( \Gamma(T^\ast(s)v_1^\ast),\dots,  \Gamma(T^\ast(s)v_n^\ast ) )\in R\big) = \gamma_s(C) \qquad\text{as }s_k\to s.
\end{split}
\end{align}
Consequently, Theorem 452C in \cite{Fremlin4} on disintegration of measures implies that
$s\mapsto \gamma_s(B)$ is measurable for all $B\in \Borel(V)$ and almost all $s\in [0,T]$, and that there exists a Borel measure $\theta$ on $\Borel(V)$ such that
\begin{align}\label{eq.measure-theta}
  \theta(B)=\int_0^T \gamma_s(B)\,ds \qquad\text{for all }B\in \Borel(V).
\end{align}
Define for each $r\in (0,\infty)$ the measure $\theta_r(B):= \theta(r^{-1/2}B)$ for all $B\in \Borel(V)$. If a sequence $(r_k)_{k\in\N}\subseteq (0,\infty)$ converges to $r\in (0,\infty)$ then it follows similarly as in \eqref{eq.cont-gamma} that $\gamma_s(r_k^{-1/2}C)\to \gamma_s(r^{-1/2}C)$ for all $s\in [0,T]$ and $C\in G$. By applying Lebesgue's theorem of dominated convergence we conclude from \eqref{eq.measure-theta} that the mapping $r\mapsto \theta_r(C)$
is continuous for all $C\in G$. Another application
of  Theorem 452C in \cite{Fremlin4} implies that there exists a measure $\mu$ on $\Borel(V)$ satisfying
\begin{align*}
  \mu(B)=\int_0^\infty \theta_r(B)\,\rho(dr)=
  \int_0^\infty \theta(r^{-1/2}B)\,\rho(dr)
  \qquad\text{for all }B\in \Borel(V).
\end{align*}
Note that by our argument above the function $(s,r)\mapsto \gamma_s(r^{-1/2}C)$ is separately continuous in both variables for all $C\in G$ and thus jointly measurable.  Tonelli's theorem enables us to conclude for all $C\in G$ that
\begin{align*}
  \mu(C)=\int_0^\infty \int_0^T \gamma_s(r^{-1/2}C)\,ds\, \rho(dr)
=\int_0^T \int_0^\infty \gamma_s(r^{-1/2}C)\, \rho(dr)\,ds
=\nu(C),
\end{align*}
which shows that $\nu$ extends to the measure $\mu$ on $\Borel(V)$. The measure $\mu$ satisfies
\begin{align*}
\int_V \Big(\norm{v}^p\wedge 1 \Big)\,\mu(dv)
&=\int_0^T \int_0^\infty \int_{H_C} \Big(\norm{T(s)(\sqrt{r}i_Ch)}^p\wedge 1 \Big)\, \gamma(dh)\rho(dr) ds\notag\\
&\le \int_0^T \int_0^\infty \left(\abs{r}^{p/2}\int_{H_C} \big(\norm{T(s)(i_Ch)}^p\big)\, \gamma(dh)\wedge 1\right) \rho(dr) ds\notag\\
&= \int_0^T \int_0^\infty \left(\abs{r}^{p/2}
\norm{T(s)\circ i_C}_{R_p(H_C,E)}^p\wedge 1\right)\, \rho(dr)\,ds.
\end{align*}
An application of Corollary \ref{co.stablecotype} completes the proof.
\end{proof}

A very similar result as Corollary \ref{co.BrzZab} is derived in \cite{BrzZab10}. However,  the conditions in our result are purely intrinsic, whereas the result in \cite{BrzZab10} is based on conditions in terms of an additional Banach space, which is not related to the problem under consideration.

\vspace{10pt}
\noindent\textbf{Acknowledgments:} the author would like to thank the referees for their careful reading and valuable comments and suggestions.

\bibliographystyle{plain}

\end{document}